\documentclass[final,1p,times]{elsarticle}
\usepackage{amssymb}
\usepackage{amsmath,amsthm,amssymb,mathrsfs,graphicx,float,appendix,subfigure,caption}
\journal{Journal of Differential Equations}

\def\a{\alpha}

\def\t{\tau}

\def\non{\nonumber}

\def\B{\Box}

\allowdisplaybreaks
\begin{document}

\newtheorem{theorem}{Theorem}[section]
\newtheorem{definition}[theorem]{Definition}
\newtheorem{lemma}[theorem]{Lemma}
\newtheorem{proposition}[theorem]{Proposition}
\newtheorem{corollary}[theorem]{Corollary}
\newtheorem{remark}[theorem]{Remark}
\renewcommand{\theequation}{\thesection.\arabic{equation}}
\catcode`@=11 \@addtoreset{equation}{section} \catcode`@=12

\begin{frontmatter}
 \title{Self-similar solutions of fast reaction limit problem with nonlinear diffusion}

\author{Elaine Crooks}

\author{Yini Du\corref{cor1}}
\ead{yini.du@swansea.ac.uk}
\cortext[cor1]{Corresponding author}

\address{Department of Mathematics, Faculty of Science and Engineering,\\ Swansea University, Swansea SA1 8EN, UK}

\begin{abstract}
In this paper, we present an approach to characterising self-similar fast-reaction limits of systems with nonlinear diffusion. For appropriate initial data, in the fast-reaction limit $k\rightarrow\infty$, spatial segregation results in the two components of the original systems converging to the positive and negative parts of a self-similar limit profile $f(\eta)$, where $\eta=\frac{x}{\sqrt t}$, that satisfies one of four ordinary differential systems. The existence of these self-similar solutions of the $k\rightarrow\infty$ limit problems is proved by using shooting methods which focus on $a$, the position of the free boundary which separates the regions where the solution is positive and where it is negative, and $\gamma$, the derivative of $-\phi(f)$ at $\eta=a$. The position of the free boundary gives us intuition about how one substance penetrates into the other, and for specific forms of nonlinear diffusion, the relationship between the given form of the nonlinear diffusion and the position of the free boundary is also studied.
\end{abstract}


\end{frontmatter}
\section{Introduction}
In this paper, we will study the self-similar characterisation of two pairs of problems, one on the half-line and the other one on the whole real line.

The first pair of problems, on the half-strip $S_T=\left\{(x,t):0\le x<\infty, 0<t<T\right\}$ and with both $\varepsilon>0$ and $\varepsilon=0$, are
\begin{align}
\left\{\begin{aligned}
&w_t=\mathcal{D}(w)_{xx},\quad &&(x,t)\in(0,\infty)\times(0,T),\\
&w(x,0)=w_0(x):=-V_0, &&{\rm for}\ x>0,\\
&w(0,t)=U_0,&&{\rm for}\ t\in (0,T).
\end{aligned}\right.
\label{c}
\end{align}
where the function $\mathcal{D}$ is defined as
\begin{align}
\mathcal{D}(s):=\left\{\begin{aligned}
&\phi(s)\quad &&s\ge0,\\
&-\varepsilon\phi(-s)\quad &&s< 0.\end{aligned}\right.
\label{c0}
\end{align}
The function $\phi\in C^2(\mathbb{R})$ and $\phi'$ are assumed to be strictly increasing with
\begin{align}
\phi(s)>0\ {\rm as}\ s>0\ {\rm and}\ \phi'(s)=\phi(s)=0\ {\rm when}\ s=0. \label{phi}
\end{align}
We will also require that $\phi$ satisfies
\begin{align}
\displaystyle\int_0^1\frac{\phi'(f)}{f}{\rm d}f<\infty\quad {\rm and}\quad \displaystyle\int_1^\infty\frac{\phi'(f)}{f}{\rm d}f=\infty.\label{phi1}
\end{align}

The function $\mathcal{D}$ arises from the $k\rightarrow\infty$ limit of the following reaction diffusion systems that have been studied in \cite{selfnon},
\begin{align}
\left\{
\begin{array}{llll}
  u_{t}=\phi(u)_{xx}-kuv,\quad &(x,t)\in(0,\infty)\times(0,T),\\
  v_{t}=\varepsilon\phi(v)_{xx}-kuv,\quad &(x,t)\in(0,\infty)\times(0,T),\\
  u(0,t)=U_{0},\quad \varepsilon \phi(v)_{x}(0,t)=0,\quad &\mbox{for}\quad t\in(0,T),\\
  u(x,0)=u^k_0(x),\quad v(x,0)=v^k_{0}(x),\quad &\mbox{for}\quad x\in\mathbb{R}^{+}.
\label{a}
\end{array}
\right.
\end{align}
Here, $u$ and $v$ represent concentration of two substances and $k$ is the positive rate constant of the reaction. As in \cite{selfsim}, the initial data for the limiting self-similar solutions are defined as
\begin{align*}
u_0^\infty=\left\{\begin{array}{llll}
&U_0\quad &x=0,\\ &0\quad &x>0,\end{array}\right.\quad
v_0^\infty=\left\{\begin{array}{llll}
&0\quad &x=0,\\ &V_0\quad &x>0,\end{array}\right.
\end{align*}
where $U_0$ and $V_0$ are positive constants.

From \cite{selfnon}, we see that the limits $u$ of $u^k$ and $v$ of $v^k$ segregate, given by the positive and negative parts respectively of a function $w$. It can be shown in \cite{selfnon}, by using a strategy inspired by \cite{wound}, that $w$ is the unique weak solution of (\ref{a}), 
\begin{align*}
u=w^+ \ {\rm and} \ v=-w^-,
\end{align*}
where $s^+=\max\{0,s\}$ and $s^-=\min\{0,s\}$. Here we prove that the limit function $w$ satisfies one of two self-similar problems, depending on whether $\varepsilon>0$ or $\varepsilon=0$. In each case, a function $f:\mathbb{R}^+\rightarrow \mathbb{R}$ describes a self-similar limit solution such that $w(x,t)=f(\eta)$ where $\eta=x/\sqrt{t}$ for $(x,t)\in S_T$. There is a free boundary at $\eta=a$ with $f(\eta)>0$ when $\eta<a$ and $f(\eta)<0$ when $\eta>a$ and the self-similar solution $f(\eta)$ satisfies the boundary conditions $f(a)=0$ and
\begin{align}
\gamma:=-\displaystyle\lim_{\eta\nearrow a}\phi'(f(\eta))f'(\eta).\label{gamma}
\end{align}
When $\varepsilon>0$, the existence of self-similar solutions is proved in Section 3 by using a two-parameter shooting methods focusing on $a$ and $\gamma$.
When $\varepsilon=0$, $\gamma$ has a specific form, namely
\begin{align*}
\gamma=\frac{aV_0}{2},
\end{align*}
and the existence of self-similar solutions is proved in Section 4 by a one-parameter shooting, since $\gamma$ depends on $a$.

The second pair of problems on the strip $Q_T=\left\{(0,t):x\in\mathbb{R},0<t<T\right\}$ with $\mathcal{D}$ from (\ref{c0}) and both with $\varepsilon>0$ and $\varepsilon=0$, are
\begin{align}
\left\{\begin{aligned}
&w_t=\mathcal{D}(w)_{xx},\quad {\rm in}\ Q_T,\\
&w(x,0)=w_0(x):=\left\{\begin{aligned}&U_0,\quad &&{\rm if}\ x<0,\\-&V_0,&&{\rm if}\ x>0.\end{aligned}\right.
\end{aligned}\right.
\label{c2}
\end{align}
The function $\mathcal{D}$ and the initial condition come from the $k\rightarrow\infty$ limit problems of
\begin{align}
\left\{
\begin{array}{llll}
  u_{t}=\phi(u)_{xx}-kuv,\quad &(x,t)\in\mathbb{R}\times(0,T),\\
  v_{t}=\varepsilon\phi(v)_{xx}-kuv,\quad &(x,t)\in\mathbb{R}\times(0,T),\\
  u(x,0)=u_0^k(x),\quad v(x,0)=v_0^k(x),\quad &\mbox{for}\quad x\in\mathbb{R},
\label{a2}
\end{array}
\right.
\end{align}
where we define, as in \cite{selfsim} that
\begin{align*}
u_0^\infty=\left\{\begin{array}{llll}
&U_0\quad &x<0,\\ &0\quad &x>0,\end{array}\right.\quad
v_0^\infty=\left\{\begin{array}{llll}
&0\quad &x<0,\\ &V_0\quad &x>0,\end{array}\right.
\end{align*}
with $U_0, V_0$ positive constants and $k$ as in (\ref{a}).

We use arguments similar to those in half-line case to prove the existence of a self-similar solution of (\ref{c2}). If $\varepsilon>0$, we may have $a<0$, $a=0$ and $a>0$ where $f(a)=0$, since $a$ is not necessarily positive in the whole-line case.

We take advantage of various ideas from earlier work on self-similar solutions for nonlinear diffusion problems and discuss briefly two previous papers \cite{simil,similarity} here.
Let $k(s)$ be continuous with $k(0)=0$ and $k(s)>0$ as $s>0$. We introduce the notation $k(s)$ for ease of comparison with \cite{simil,similarity}, but, $k$ clearly plays the same role as $\phi'$ plays here. Then the solutions of the nonlinear diffusion equation $u_t=(k(u)_x)_x$ can be studied in self-similar form $u(x,t)=f(\eta)$ where $\eta=x/\sqrt{t}$, and $f$ satisfies the equation
\begin{align}
\left(k(f)f'\right)'+\frac{1}{2}\eta f'=0.
\label{st1}
\end{align}

 In \cite{simil}, Atkinson and Peletier proved the existence and uniqueness of a self-similar solution $f(\eta)$ which satisfies (\ref{st1}) for $0<\eta<a$, where $a>0$, under the boundary conditions $f(0)=U$, $\displaystyle\lim_{\eta\rightarrow a}f(\eta)=0$ and $\displaystyle\lim_{\eta\rightarrow a}k(f(\eta))f'(\eta)=0$. They consider two cases in describing the dependence of $a$ and $U$,
 \begin{align*}
A. \displaystyle\int_1^\infty\frac{k(s)}{s}{\rm d}s=\infty \qquad {\rm and} \qquad  B. \displaystyle\int_1^\infty\frac{k(s)}{s}{\rm d}s<\infty.
 \end{align*}
 They found that as $U\rightarrow\infty$, $a=a(U)$ tends to infinity in Case A whereas $a(U)$ tends to a finite limit in Case B.
   Here, we only consider the case when $\displaystyle\int_1^\infty\frac{k(s)}{s}{\rm d}s=\infty$. The proof in \cite{simil} depends on a discussion of the function $b(a)$, which is defined as the value at $\eta=0$ of the solution $f(\eta)=f(\eta;a)$ of (\ref{st1}) with boundary conditions $\displaystyle\lim_{\eta\rightarrow a}f(\eta)=0$ and $\displaystyle\lim_{\eta\rightarrow a}k(f(\eta))f'(\eta)=0$. A similar problem in an unbounded interval $0<\eta<\infty$ with boundary conditions $f(0)=U$ and $\displaystyle\lim_{\eta\rightarrow\infty}f(\eta)=0$ is studied in \cite{similarity} by Craven and Peletier. Note that in \cite{similarity}, $f(\eta)>0$ for all $\eta>0$. \cite{similarity} proved the existence and uniqueness of a weak solution by a shooting method where the initial value problem $f(0)=U$, $f'(0)=\beta$ is considered. We will adapt ideas of studying the function $b(a)$ from \cite{simil} and shooting methods from \cite{similarity} to prove existence of self-similar solutions in this paper.

Note that \cite{simil,similarity} treated a single equation where the solutions were always non-negative. In this paper, we have sign-changing solutions since the free boundary separates regions where the solutions are positive and where the solutions are negative. Here, our self-similar solutions satisfy a certain equation when they are positive, and a different equation where they are negative. We exploit ideas from \cite{simil,similarity} to investigate our self-similar limit problems that involve these two equations.

 We will study the existence of a solution $f$ that satisfies (\ref{j}) with given boundary conditions by splitting it into two parts: $\eta<a$ where $f(\eta)$ is positive, and $\eta>a$ where $f(\eta)$ is negative. Then we will discuss the existence and properties of $\displaystyle\lim_{\eta\rightarrow 0}f(\eta)$ and $\displaystyle\lim_{\eta\rightarrow\infty}f(\eta)$. These results will be used to study $b(a,\gamma)$, the value at $\eta\rightarrow0$ of the solution $f(\eta)=f(\eta;a,\gamma)$, and $d(a,\gamma)$, the value at $\eta\rightarrow\infty$ of the solution $f(\eta)=f(\eta;a,\gamma)$, where $\gamma$ as in (\ref{gamma}), and also to implement a two-parameter shooting method.

This paper is organised as follows. In Section 2, the limit problem (\ref{c}) is characterised as a self-similar solution of the problem first in Theorem {\ref{tj}} when $\varepsilon>0$, and then in Theorem {\ref{tjj}} when $\varepsilon=0$. Section 3 focuses on properties of the parameters $a$ and $\gamma$ in the study of the self-similar solution $f$, and prove some preliminary results that are useful in deducing existence of self-similar solutions. The existence of self-similar solutions when $\varepsilon\ge0$ is proved in Section 3.5 and Section 4. Section 5 contains the whole-line counterparts of the study of the half-line problem in Sections 2-4.

In Section 6, we also consider a specific family of $\phi'(f)=f^{m-1}$ where $m>2$ is a constant, and investigate how the free boundary position $a$ is affected by $m$. Note that with fixed $U_0,V_0$, there exists a unique self-similar solution which determines $a$ and $\gamma$. We prove some further results under the additional conditions that $U_0<1$ and $m\ge2$. In particular, if $\varepsilon=0$, we find that if $m_1>m_2$, then $a_{m_1}<a_{m_2}$ which is proved in Theorem \ref{tend}. This result indicates that when $m$ gets smaller, one substance penetrates into the other faster, which can also be seen from the numerical result in \cite{onclass}.

\noindent{\bf Acknowledgement }

 {The authors gratefully acknowledge funding from the EPSRC EP/W522545/1. This paper is based on part of corresponding author's Ph.D thesis at Swansea University}

\section{Half-line case: preliminaries}
First, we give the definition of the weak solution of problem (\ref{c}). Note that the uniqueness of the weak solution of (\ref{c}) is proved in \cite{selfnon}.
\begin{definition}
A function $w$ is a weak solution of (\ref{c}) if
\begin{itemize}
\item[{\rm(i)}]$w\in L^\infty(S_T)$,
\item[{\rm(ii)}]$\mathcal{D}(w)\in\mathcal{D}(\hat w)+L^{2}(0,T;W^{1,2}_0(\mathbb{R}^+))$, where $\hat w\in C^\infty(\mathbb{R}^+)$ is a smooth function with $\hat w=U_0$ when $x=0$ and $\hat w=-V_0$ when $x>1$,
\item[{\rm(iii)}]$w$ satisfies
\begin{align}
\int_{\mathbb{R}^+}w_0(x)\xi(x,0){\rm d}x+\iint_{S_T}w\xi_t{\rm d}x{\rm d}t=\iint_{S_T}\mathcal{D}(w)_x\xi_x{\rm d}x{\rm d}t.
\label{au}
\end{align}
for all $\xi\in\mathcal{F}_T:=\left\{\xi\in C^1(S_T):\ \xi(0,t)=\xi(\cdot,T)=0\  {\rm for}\ t\in(0,T)\ {\rm and}\ {\rm supp}\,\xi\subset[0,J]\times[0,T]\right.\\ \left.  {\rm for}\ {\rm some}\ J>0 \right\}.$
\end{itemize}
\label{au1}
\end{definition}

We state a free-boundary problem, including interface conditions, that is satisfied by the solution $w$
 of (\ref{c}) under some regularity assumptions and conditions on the form of the free boundary.
The following result follows from a similar approach to that of \cite[Theorem 5]{sing}. We sketch the key points here, focusing on the parts where our problem needs a slightly different argument.
\begin{theorem}
Let $w$ be the unique weak solution of problem (\ref{c}). Suppose that there exists a function $\beta:[0,T]\rightarrow\mathbb{R^+}$ such that for each $t\in[0,T]$,
\begin{align*}
w(x,t)>0\ {\rm if}\ x<\beta(t)\quad {\rm and}\quad w(x,t)<0\ {\rm if}\ x>\beta(t).
\end{align*}
Then if $t\mapsto\beta(t)$ is sufficiently smooth and the functions $u:=w^+$ and $v:=-w^-$ are smooth up to $\beta(t)$, the functions $u,v$ satisfy
\begin{align}
\left\{\begin{aligned}
&u_t=\phi(u)_{xx},\quad &&{\rm in}\ (x,t)\in S_T:x<\beta(t),\\
&v_t=\varepsilon\phi(v)_{xx},\quad &&{\rm in}\ (x,t)\in S_T:x>\beta(t),\\
&\langle\phi(u)\rangle=\varepsilon\langle\phi(v)\rangle=0,\quad &&{\rm on}\ \Gamma_T:=\left\{(x,t)\in S_T:x=\beta(t)\right\},\\
&\langle v\rangle\beta'(t)=\langle\phi(u)_x-\varepsilon\phi(v)_x\rangle,\quad &&{\rm on}\ \Gamma_T:=\left\{(x,t)\in S_T:x=\beta(t)\right\},\\
&u=U_0,\quad &&{\rm on}\ \left\{0\right\}\times[0,T],\\
&u(\cdot,0)=u_0^\infty(\cdot),\ v(\cdot,0)=v_0^\infty(\cdot),\quad &&{\rm in}\ \mathbb{R}^+,
\end{aligned}\right.
\label{e}
\end{align}
where $\langle\cdot\rangle$ denotes the jump across $\beta(t)$ from $\left\{x<\beta(t)\right\}$ to $\left\{x>\beta(t)\right\}$,
\begin{align*}
\langle \alpha\rangle:=\displaystyle\lim_{x\searrow\beta(t)}\alpha(x,t)-\displaystyle\lim_{x\nearrow\beta(t)}\alpha(x,t),
 \end{align*}
and $\beta'(t)$ denotes the speed of propagation of the free boundary $\beta(t)$.
\label{tl}
\end{theorem}
\begin{proof}
We recall that $(u,v)$ satisfies
\begin{align*}
\iint_{S_T}(u-v)\xi_t{\rm d}x{\rm d}t+\int_{\mathbb{R}^+}(u_0^\infty-v_0^\infty)\xi(x,0){\rm d}x=\iint_{S_T}(\phi(u)-\varepsilon\phi(v))_x\xi_{x}{\rm d}x{\rm d}t,
\end{align*}
for all $\xi\in\mathcal{F}_T$, from which we have
\begin{align}
\langle-u+v\rangle\beta'(t)+\langle-\phi(u)_x+\varepsilon\phi(v)_x\rangle=0\quad {\rm on}\ \Gamma_T:=\left\{(x,t)\in S_T:x=\beta(t)\right\}.\label{fb2}
\end{align}

Now we know that $\mathcal{D}(w)$ is a continuous function of $x$ for almost every $t\in[0,T]$, since $\mathcal{D}(w)\in\mathcal{D}(\hat w)+L^2(0,T;W^{1,2}(\mathbb{R}^+))$ by Definition \ref{au1} (ii). So $\langle\mathcal{D}(w)\rangle=0$, which implies
\begin{align*}
-\lim_{x\searrow\beta(t)}\varepsilon\phi(-w^-)-\lim_{x\nearrow\beta(t)}\phi(w^+)=-\lim_{x\searrow\beta(t)}\varepsilon\phi(v)-\lim_{x\nearrow\beta(t)}\phi(u)=0.
\end{align*}
Therefore we get
\begin{align}
\langle\phi(u)\rangle=\varepsilon\langle\phi(v)\rangle=0.\label{tl3}
\end{align}
Moreover, since $\phi\in C^2(\mathbb{R})$ is strictly increasing, $u(\cdot,t)$ is continuous across $\beta(t)$ and if $\varepsilon>0$, $v(\cdot,t)$ is also continuous across $\beta(t)$, so that
\begin{align}
&\langle u\rangle=0\quad {\rm if}\ \varepsilon\ge0,\label{tl1} \\ &\langle v\rangle=0\quad {\rm if}\ \varepsilon>0.\label{tl2}
\end{align}

Then (\ref{fb2}) and the fact that $\langle u\rangle=0$ imply that
\begin{align}
\langle v\rangle\beta'(t)=\langle\phi(u)_x-\varepsilon\phi(v)_x\rangle,\quad {\rm on}\ \Gamma_T.\label{tl4}
\end{align}
\end{proof}

The following two limit problems are obtained by interpreting the interface  conditions on $\beta(t)$.
\begin{corollary}
Let $w$ and $\beta:[0,t]\rightarrow\mathbb{R^+}$ satisfy the hypotheses of theorem \ref{tl}. Then the functions $u:=w^+$, $v:=-w^-$ satisfy one of limit problems depending on whether $\varepsilon>0$ or $\varepsilon=0$. If $\varepsilon>0$, then
\begin{align}
\left\{\begin{aligned}
&u_t=\phi(u)_{xx},\quad &&{\rm in}\ \left\{(x,t)\in S_T:x<\beta(t)\right\},\\
&v=0,\quad &&{\rm in}\ \left\{(x,t)\in S_T:x<\beta(t)\right\},\\
&v_t=\varepsilon\phi(v)_{xx},\quad &&{\rm in}\ \left\{(x,t)\in S_T:x>\beta(t)\right\},\\
&u=0,\quad &&{\rm in}\ \left\{(x,t)\in S_T:x>\beta(t)\right\},\\
&\lim_{x\nearrow\beta(t)}u(x,t)=0=\lim_{x\searrow\beta(t)}v(x,t)\quad &&{\rm for}\ {\rm each}\ t\in[0,T],\\
&\lim_{x\nearrow\beta(t)}\phi[u(x,t)]_x=-\varepsilon\lim_{x\searrow\beta(t)}\phi[v(x,t)]_x\quad &&{\rm for}\ {\rm each}\ t\in[0,T],\\
&u=U_0,\quad &&{\rm on}\ \left\{0\right\}\times[0,T],\\
&u(\cdot,0)=u_0^\infty(\cdot),\ v(\cdot,0)=v^\infty_0(\cdot),\quad &&{\rm in}\ \mathbb{R}^+,
\end{aligned}\right.
\end{align}
whereas if $\varepsilon=0$ and we suppose additionally that $\beta(0)=0$ and $t\mapsto\beta(t)$ is a non-decreasing function, then
\begin{align}
\left\{\begin{aligned}
&u_t=\phi(u)_{xx},\quad &&{\rm in}\ \left\{(x,t)\in S_T:x<\beta(t)\right\},\\
&v=0,\quad &&{\rm in}\ \left\{(x,t)\in S_T:x<\beta(t)\right\},\\
&v=V_0,\quad &&{\rm in}\ \left\{(x,t)\in S_T:x>\beta(t)\right\},\\
&u=0,\quad &&{\rm in}\ \left\{(x,t)\in S_T:x>\beta(t)\right\},\\
&\lim_{x\nearrow\beta(t)}u(x,t)=0\quad &&{\rm for}\ {\rm each}\ t\in[0,T],\\
&V_0\beta'(t)=-\lim_{x\nearrow\beta(t)}\phi[u(x,t)]_x\quad &&{\rm for}\ {\rm each}\ t\in[0,T],\\
&u=U_0,\quad &&{\rm on}\ \left\{0\right\}\times[0,T],\\
&u(\cdot,0)=u_0^\infty(\cdot),\ v(\cdot,0)=v^\infty_0(\cdot),\quad &&{\rm in}\ \mathbb{R}^+,
\end{aligned}\right.
\end{align}
where $\beta'(t)$ denotes the speed of propagation of the free boundary $\beta(t)$.
\label{clfb}
\end{corollary}
\begin{proof}We have $\langle u\rangle=0$ from (\ref{tl1}). From (\ref{tl2}), we know that if $\varepsilon>0$, $\langle v\rangle=0$,
whereas if $\varepsilon=0$, $v(\cdot,t)$ jumps across $\beta(t)$. The fact that $v_t=0$ in $\left\{(x,t)\in S_T:x>\beta(t)\right\}$ together with the initial condition that $v_0^\infty(x)=V_0$ if $x>0$ give the result that $v(x,t)=V_0$ for all $x\ge\beta(t)$, since $\beta(0)=0$ and $t\mapsto\beta(t)$ is a non-decreasing function. It follows that if $\varepsilon=0$,
\begin{align*}
\langle v\rangle=V_0-0=V_0\ {\rm for}\ {\rm all}\ t\in[0,T].
\end{align*}
The normal derivative condition (\ref{tl4}) implies that if $\varepsilon>0$, then $\langle\phi(u)_x-\varepsilon\phi(v)_x\rangle=0$, so that
\begin{align*}
\lim_{x\nearrow\beta(t)}\phi[u(x,t)]_x=-\varepsilon\lim_{x\searrow\beta(t)}\phi[v(x,t)]_x.
\end{align*}
On the other hand, if $\varepsilon=0$, then
\begin{align*}
V_0\beta'(t)=-\displaystyle\lim_{x\nearrow\beta(t)}\phi(u)_x.
\end{align*}
\end{proof}

Next we will prove that if we have a self-similar solution of (\ref{j}), then it is a weak solution of (\ref{c}) in the sense of Definition \ref{au1} and then prove the existence of the self-similar solution of (\ref{j}) by a two-parameter method in Section 3.5.
\begin{theorem}\label{tj}
The unique weak solution $w$ of problem (\ref{c}) with $\varepsilon>0$ has a self-similar form. There exists a function $f:\mathbb{R^+}\mapsto\mathbb{R^+}$ and a constant $a\in\mathbb{R^+}$ such that
\begin{align*}
w(x,t)=f\left(\frac{x}{\sqrt{t}}\right),\ (x,t)\in S_T\ {\rm and}\ \beta(t)=a\sqrt{t},\ t\in[0,T].
\end{align*}
 Denoting $\eta=\displaystyle\frac{x}{\sqrt{t}}$, $f$ satisfies the system
\begin{align}
\left\{\begin{aligned}
  &-\frac{1}{2}\eta f'(\eta)=[\phi'(f(\eta))f'(\eta)]',\quad &&{\rm if}\ \eta<a,\\
  &-\frac{1}{2}\eta f'(\eta)=[\varepsilon\phi'(-f(\eta))f'(\eta)]',\quad &&{\rm if}\ \eta>a,\\
  &f(0)=U_0,\quad \lim_{\eta\rightarrow \infty}f(\eta)=-V_0,\\
  &\lim_{\eta\nearrow a}f(\eta)=0=-\lim_{\eta\searrow a}f(\eta),\\
  &\lim_{\eta\nearrow a}\phi'(f(\eta))f'(\eta)=\varepsilon\lim_{\eta\searrow a}\phi'(-f(\eta))f'(\eta),\\
\end{aligned}\right.
\label{j}
\end{align}
where a prime denotes differentiation with respect to $\eta$.
\end{theorem}

\begin{proof}We know that from \cite[Lemma 4.7]{selfnon} if $w(x,t)=f(\frac{x}{\sqrt{t}})$ is a weak solution of (\ref{c}), then it is unique. We therefore need to show that a solution to (\ref{j}) exists, which will be postponed to Section 3.3, and that if $f$ satisfies (\ref{j}), then it is a weak solution of (\ref{c}), that is, it satisfies Definition \ref{au1}.

The weak solution of (\ref{c}) satisfies
\begin{align*}
\iint_{S_T}w\xi_t{\rm d}x{\rm d}t-\iint_{S_T}\mathcal{D}(w)_x\xi_x{\rm d}x{\rm d}t=V_0\int_{\mathbb{R}^+}\xi(x,0){\rm d}x,
\end{align*}
where $\xi\in\mathcal{F}_T$. If we write $w(x,t)=f(\frac{x}{\sqrt{t}})$, then since $\xi(\cdot,T)=0$, with simple calculations we can prove that $f$ satisfies Definition \ref{au1} (iii).

By some calculations which use the properties of $f$ shown in Lemmas \ref{mono}-\ref{ld0} below, it can be shown that $\phi(U_0)-\phi(f)\in L^2((0,a))$,  $\phi(V_0)-\phi(-f)\in L^2((a,\infty))$ and $-(\phi(-f))'\in L^2((a,\infty))$, see \cite{thesis} for details. It therefore follows that if $f$ satisfies (\ref{j}), then by changing variables, it satisfies Definition \ref{au1} (ii), that is
\begin{align*}
\mathcal{D}(f)\in\mathcal{D}(\hat w)+L^2(0,T;W^{1,2}_0(\mathbb{R}^+)).
\end{align*}
Hence $f$ satisfies Definition \ref{au1}. It remains to prove the existence of solution of Problem (\ref{j}), which is done in Theorem \ref{twe1} below.
\end{proof}

Recall that
\begin{align}
\gamma:=-\displaystyle\lim_{\eta\nearrow a}\phi'(f(\eta))f'(\eta).\label{gam}
\end{align}
Note that it follows from the free boundary condition of Problem (\ref{j}) that
\begin{align*}
\gamma=-\varepsilon\displaystyle\lim_{\eta\searrow a}\phi'(-f(\eta))f'(\eta)\quad {\rm when}\ \varepsilon>0.
\end{align*}

We next prove a collection of properties of $f$ that are both need in the proof of Theorem \ref{tj} and useful later.
First, we prove the monotonicity of $f$.
\begin{lemma}\label{mono}
Suppose $\varepsilon>0$. If $f$ satisfies (\ref{j}), then $f'(\eta)<0$ for all $\eta\ne a$.
\end{lemma}
\begin{proof}Suppose $f$ is not monotonic, then there exists $\eta_0\ne a$ such that $f'(\eta_0)=0$, denote $f_0:=f(\eta_0)\ne0$. Then defining the function by $g:\mathbb{R}\rightarrow\mathbb{R}^+$
\begin{align*}
g(\eta)=f_0>0,\quad {\rm for}\ {\rm all}\ \eta\in\mathbb{R},
\end{align*}
we have that $g$ satisfies
\begin{align*}
-\frac{1}{2}\eta g'(\eta)=[\phi'(g(\eta))g'(\eta)]'\\
g(\eta_0)=f_0,\quad g'(\eta_0)=0,
\end{align*}
but also
\begin{align*}
-\frac{1}{2}\eta f'(\eta)=[\phi'(f(\eta))f'(\eta)]'\\
f(\eta_0)=f_0,\quad f'(\eta_0)=0.
\end{align*}
By Picard's theorem and the uniqueness, it follows that $f=g$ either for all $\eta<a$ or $\eta>a$, depending on whether $f(\eta_0)>0$ or $f(\eta_0)<0$. But this contradict the boundary conditions in (\ref{j}), so we know that $f$ must be monotonically decreasing.
\end{proof}

We prove next that $\gamma$ defined in (\ref{gam}) is strictly positive when $\varepsilon>0$. In fact, $\gamma$ is also strictly positive when $\varepsilon=0$, see Corollary \ref{gam0}.
\begin{lemma}
Suppose $\varepsilon>0$. Let $f$ be a solution of (\ref{j}), then $\gamma>0$.\label{ga2}
\end{lemma}
\begin{proof}If $\gamma\le0$, then integrating the equation for $\eta>a$ in (\ref{j}) from $a$ to $\eta$ yields
\begin{align}
-\frac{1}{2}\int_a^\eta sf'(s){\rm d}s=\varepsilon\phi'(-f(\eta))f'(\eta)+\gamma.\label{ga1}
\end{align}
The left-hand side of (\ref{ga1}) is positive since $f'<0$ by Lemma \ref{mono} whereas the right-hand side of (\ref{ga1}) is negative since $f'<0$ and $\gamma\le0$. Therefore, it follows that $\gamma>0$ by the contradiction.
\end{proof}

Now we prove some further properties of $f'$.

\begin{lemma}
Suppose $\varepsilon>0$. If $f$ satisfies (\ref{j}), then for $\eta>a$, $f'$ is monotonically increasing in $\eta$.
\end{lemma}
\begin{proof}The results follow from the equation for $f$ when $\eta>a$
\begin{align*}
-\frac{1}{2}\eta f'(\eta)=-\varepsilon\phi''(-f(\eta))(f'(\eta))^2+\varepsilon\phi'(-f(\eta))f''(\eta).
\end{align*}
The left-hand side is positive since $0<a<\eta$ and the first term of right-hand side is negative. Therefore $f''$ must be positive.
\end{proof}

\begin{lemma}\label{ld0}
Suppose $\varepsilon>0$. If $f$ satisfies (\ref{j}), then we have for $\eta>a$
\begin{align}
-f'(\eta)\le\displaystyle\frac{4\gamma}{\eta^2-a^2},
\label{fd}
\end{align}
and hence, in particular,
\begin{align}
\displaystyle\lim_{\eta\rightarrow\infty}f'(\eta)=0.\label{f0}
\end{align}
\end{lemma}
\begin{proof}Integrating the equation of $f$ for $\eta>a$ from $a$ to $\eta$, we have
\begin{align*}
-\frac{1}{4}f'(\eta)(\eta^2-a^2)\le-\frac{1}{2}\int_a^\eta s f'(s){\rm d}s=\varepsilon\phi'(-f(\eta))f'(\eta)+\gamma\le\gamma,
\end{align*}
since $f'<0$ is increasing and the right-hand side is positive.

Therefore we obtain
\begin{align*}
-f'(\eta)\le\displaystyle\frac{4\gamma}{\eta^2-a^2}.
\end{align*}
If we choose $\eta>a+1$ then
\begin{align}
-f'(\eta)<\frac{4\gamma}{2\eta-1},\label{fpr1}
\end{align}
which vanishes as $\eta\rightarrow\infty$.
\end{proof}

Similarly, we can prove that if we have a self-similar solution of (\ref{jjj}) when $\varepsilon=0$, then it is a weak solution of (\ref{c}) in the sense of Definition \ref{au1}, the existence of the self-similar solution is proved by using one-parameter shooting in Theorem \ref{t01}.
\begin{theorem}
The unique weak solution $w$ of problem (\ref{c}) with $\varepsilon=0$ has a self-similar form. There exists a function $f:\mathbb{R^+}\mapsto\mathbb{R^+}$ and a constant $a\in\mathbb{R^+}$ such that
\begin{align*}
w(x,t)=f\left(\frac{x}{\sqrt{t}}\right),\ (x,t)\in S_T\ {\rm and}\ \beta(t)=a\sqrt{t},\ t\in[0,T].
\end{align*}
 Denote $\eta=\displaystyle\frac{x}{\sqrt{t}}$, $f$ satisfies the system
\begin{align}
\left\{\begin{aligned}
  &-\frac{1}{2}\eta f'(\eta)=[\phi'(f(\eta))f'(\eta)]',\quad &&{\rm if}\ \eta<a,\\
  &f(\eta)=-V_0,\quad &&{\rm if}\ \eta>a,\\
  &f(0)=U_0,\\
  &\lim_{\eta\nearrow a}f(\eta)=0,\\
  &\lim_{\eta\nearrow a}\phi'(f(\eta))f'(\eta)=-\frac{aV_0}{2},\\
\end{aligned}\right.
\label{jjj}
\end{align}
where a prime denotes differentiation with respect to $\eta$.
\label{tjj}
\end{theorem}
\begin{proof}When $\varepsilon=0$, we can only consider $f$ for $\eta<a$, since $f(\eta)=-V_0$ for $\eta>a$. The proof of the fact that a solution of (\ref{jjj}) yields a weak solution of (\ref{c}) is similar to the proof of Theorem \ref{tj} and the existence of solution for problem (\ref{jjj}) is proved in Theorem \ref{t01} below by a shooting method.\end{proof}

The following corollary follows from the fact that $\gamma=\displaystyle\frac{aV_0}{2}$ when $\varepsilon=0$, which is a direct consequence of the free boundary conditions in (\ref{jjj}).
\begin{corollary}
Suppose $\varepsilon=0$. Let $f$ be a solution of (\ref{jjj}), then $\gamma=\displaystyle\frac{aV_0}{2}$ is positive.\label{gam0}
\end{corollary}

\section{Half-line case: self-similar solutions with $\varepsilon>0$}
In this section, we will prove the existence of the self-similar solution by splitting the proof into two parts: $\eta< a$ where $f(\eta)>0$, and $\eta> a$ where $f(\eta)<0$. We will discuss the existence and properties of $\displaystyle\lim_{\eta\rightarrow 0}f(\eta)$ and $\displaystyle\lim_{\eta\rightarrow \infty}f(\eta)$ and then use a two parameter shooting argument with parameters $a$ and $\gamma$.

First we consider $f$ that satisfies the equation
\begin{align}
-\frac{1}{2}\eta f'(\eta)=[\phi'(f(\eta))f'(\eta)]',\quad  0<\eta<a.
\label{k}
\end{align}
At the boundaries we seek a solution that satisfies
\begin{align}
&f(0)=U_0,\label{kb2}\\
&\displaystyle\lim_{\eta\nearrow a}f(\eta)=0,\quad \displaystyle\lim_{\eta\nearrow a}\phi'(f(\eta))f'(\eta)=-\gamma.\label{kb1}
\end{align}

\subsection{Solution in left-neighbourhood of $\eta=a$}
We start by proving the local existence of a positive solution of (\ref{k}) in a left-neighbourhood of $\eta=a$, which satisfies the boundary conditions (\ref{kb1}).

\begin{lemma}
Suppose $a>0$ and choose $ a_1< a$. If $f$ satisfies (\ref{k}) and the boundary conditions (\ref{kb1}), then
\begin{align}
\int_0^{f(a_1)}\frac{\phi'(f)}{\frac{2\gamma}{a}+f}{\rm d}f\le \frac{1}{2}a^2.
\label{pf1}
\end{align}
\end{lemma}
\begin{proof}This follows from a similar argument to that in the proof of \cite[Lemma 3]{simil}.
Integration of (\ref{k}) from $\eta$ to $a$ gives 
\begin{align*}
|\phi'(f(\eta))f'(\eta)|\le\gamma-\frac{1}{2}a\int_\eta^af'(s){\rm d}s,
\end{align*}
since $a>0$ and $f'<0$ by Lemma \ref{mono}. Thus for any $a_2$ such that $a_1<a_2<a$, we have
\begin{align*}
\int_f(a_2)^{f(a_1)}\frac{\phi'(f)}{\frac{2\gamma}{a}+f}{\rm d}f\le \frac{1}{2}(a_2-a_1)\le\frac{1}{2}a^2.
\end{align*}
Then the result follows by letting $a_2$ tends to $a$.

Note that here we have an extra positive term $\frac{2\gamma}{a}$ on the denominator, which make the proof easier than in \cite{simil}, which corresponds to $\gamma=0$. \end{proof}

In the following, we prove the existence and uniqueness of the local solution by using a method inspired by \cite[Lemmas 4 and 5]{simil}.
\begin{lemma}\label{lkb1}
For given $a$ and $\gamma$, there exists  $\delta>0$ such that for $\eta\in(a-\delta,a)$, equation (\ref{k}) has a unique solution which is positive and satisfies the boundary condition (\ref{kb1}).
\end{lemma}
\begin{proof}It is convenient to start by supposing that such a solution exists in a left-neighbourhood of $\eta=a$. Integrating (\ref{k}) from $\eta$ to $a$, we have
\begin{align}
\frac{1}{f'(\eta)}=\frac{2\phi'(f)}{\int_{\eta}^asf'(s){\rm d}s-2\gamma}.\label{ba4}
\end{align}
Since $f$ is monotonic, with non-vanishing derivative, we can treat $\eta$ as a function of $f$, writing $\eta=\sigma(f)$. Then (\ref{ba4}) takes the form
\begin{align*}
\frac{{\rm d}\sigma}{{\rm d}f}=\frac{-2\phi'(f)}{\int_0^f\sigma(s){\rm d}s+2\gamma},
\end{align*}
and $\sigma(f)$ is a solution of this integro-differential equation which satisfies the initial condition $\sigma(0)=a$ and is defined and continuous on an interval $[0,\hat f]$ for some $\hat f>0$ and continuously differentiable on $(0,\hat f)$. An integration gives
\begin{align}
\sigma(f)=a-2\int_0^f\frac{\phi'(\theta)}{\int_0^\theta \sigma(s){\rm d}s+2\gamma}{\rm d}\theta,\label{ba5}
\end{align}
and if we set
\begin{align*}
\tau(f)=1-\frac{\sigma(f)}{a}=1-\frac{\eta}{a},
\end{align*}
then (\ref{ba5}) becomes
\begin{align}
\tau(f)=2a^{-2}\int_0^f\frac{\phi'(\theta)}{\int_0^\theta [1-\tau(s)]{\rm d}s+\frac{2\gamma}{a}}{\rm d}\theta.\label{ba6}
\end{align}
If the solution of (\ref{ba6}) is unique, the corresponding solution of equation (\ref{k}) is also unique.

Now we prove (\ref{ba6}) has a unique solution on $[0,\mu]$ for some $\mu>0$.
\begin{lemma}
There exists $\mu>0$ such that (\ref{ba6}) has a unique continuous solution in $0\le f\le\mu$, which is such that $\t(0)=0$ and $\t(f)>0$ if $0<f\le\mu$.
\end{lemma}
\begin{proof}With $\mu$ to be chosen later, we denote by $X$ the set of continuous functions $\t(f)$ defined on $[0,\mu]$, satisfying $0\le\t(f)\le\frac{1}{2}$. We denote by $\Vert\centerdot\Vert$ the supremum norm on $X$. Then $X$ is a complete metric space. On $X$ we introduce the map
\begin{align*}
M(\t)(f)=2a^{-2}\int_0^f\frac{\phi'(\theta)}{\int_0^\theta [1-\tau(s)]{\rm d}s+\frac{2\gamma}{a}}{\rm d}\theta
\le4a^{-2}\int_0^\mu\frac{\phi'(\theta)}{\theta+\frac{4\gamma}{a}}{\rm d}\theta.
\end{align*}
It is clear that $M(\t)(f)$ is well-defined, non-negative and continuous. Moreover, $M(\t)(f)\le\frac{1}{2}$ if
\begin{align}
4a^{-2}\int_0^\mu\frac{\phi'(\theta)}{\theta+\frac{4\gamma}{a}}{\rm d}\theta\le\frac{1}{2}.\label{ba7}
\end{align}
Therefore, if $\mu$ is chosen small enough that (\ref{ba7}) is satisfied, $M$ maps $X$ into itself.

We also wish to ensure that $M$ is a contraction map. Let $\t_1,\t_2\in X$, we have
\begin{align*}
\Vert M(\t_1)-M(\t_2)\Vert\le&2a^{-2}\int_0^f\phi'(\theta)\frac{\int_0^\theta|\t_1(s)-\t_2(s)|{\rm d}s}{\left\{\int_0^\theta [1-\t_1(s)]{\rm d}s+\frac{2\gamma}{a}\right\}\left\{\int_0^\theta [1-\t_2(s)]{\rm d}s+\frac{2\gamma}{a}\right\}}\\ \le&8a^{-2}\int_0^\mu\phi'(\theta)\frac{\theta}{\left(\theta+\frac{4\gamma}{a}\right)^2}{\rm d}\theta\,\Vert\t_1-\t_2\Vert\\
\le&8a^{-2}\int_0^\mu\frac{\phi'(\theta)}{\theta+\frac{4\gamma}{a}}{\rm d}\theta\,\Vert\t_1-\t_2\Vert,
\end{align*}
and it follows that $M$ is a contraction map if
\begin{align*}
8a^{-2}\int_0^\mu\frac{\phi'(\theta)}{\theta+\frac{4\gamma}{a}}{\rm d}\theta<1.
\end{align*}
This constitutes our second restriction on $\mu$, it clearly implies the first one, (\ref{ba7}). The result now follows from the standard fixed-point principle \cite{line}. \end{proof}
This concludes the proof of Lemma \ref{lkb1}.\end{proof}

For any $a>0$, the unique positive solution $f(\eta)$, defined in a left-neighbourhood of $\eta=a$, which satisfies the boundary conditions (\ref{kb1}), may be uniquely continued backward as a function of $\eta$. By Lemma \ref{mono}, it will increase monotonically as $\eta$ decreases. There are then two possibilities, either the solution can be continued back to $\eta=0$, or else we have $f(\eta)\rightarrow\infty$ as $\eta$ decreases towards some non-negative value. We now show that the solution can indeed be continued back to $\eta=0$. Note that the following result uses the condition (\ref{phi1}).

\begin{lemma}
For any given $a,\gamma$, the unique local solution of equation (\ref{k}) in Lemma \ref{lkb1} can be continued back to $\eta=0$.
\label{lbw}
\end{lemma}
\begin{proof}Suppose $0\le a_1<a$ and $f(\eta)\rightarrow\infty$ as $\eta\rightarrow a_1$. If there exist $a_2\in(a_1,a)$ is such that $f(a_2)>\frac{2\gamma}{a}$, then we have from (\ref{phi1}) that
\begin{align}
\int_{f(a_2)}^\infty\frac{\phi'(f)}{f+\frac{2\gamma}{a}}{\rm d}f>\frac{1}{2}\int_{f(a_2)}^\infty\frac{\phi'(f)}{f}{\rm d}f=\infty.\label{kb5}
\end{align}
But the boundedness of the integral from (\ref{pf1}), together with (\ref{kb5}), implies the boundedness of $f(a_1)$.

Now consider $a_1\le\eta\le a-\delta$ for $\delta>0$. Integrating (\ref{k}) from $\eta$ to $a-\delta$ yields
\begin{align*}
-f'(\eta)=\frac{1}{\phi'(f(\eta))}\left(\gamma-\frac{1}{2}(a-\delta)f(a-\delta)+\eta f(\eta)+\int_\eta^{a-\delta}f(s){\rm d}s\right),
\end{align*}
which implies for some constant $C$ that $-f'(\eta)\le C$ for $\eta\le a-\delta$. It follows from \cite[Theorem 1.186]{odea} that the solution can be continued back to $\eta=0$.\end{proof}

\subsection{Properties of $b(a,\gamma)$}
Define
\begin{align*}
b(a,\gamma)=\displaystyle\lim_{\eta\rightarrow0}f(\eta;a,\gamma),
\end{align*}
where $\gamma:=-\displaystyle\lim_{\eta\nearrow a}\phi'(f(\eta))f'(\eta)$
with $\gamma>0$. The following discussions on $b(a,\gamma)$ are used in proving existence of self-similar solution by shooting from $\eta=a$ with a given choice of $\gamma$, the derivative of $\phi(f)$ at $\eta=a$, back to $\displaystyle\lim_{\eta\rightarrow0}f(\eta;a,\gamma)$.

\begin{lemma}\label{lbg}
$b(a,\gamma)$ has the following properties with fixed $a$:
\begin{itemize}
\item[{\rm (i)}] $b(a,\gamma)$ is strictly monotonically increasing in $\gamma$;
\item[{\rm (ii)}] $b(a,\gamma)$ is a continuous function of $\gamma$ and the Lipschitz constant is uniform in $\gamma\in[\gamma_0,\gamma_3]$, where $0\le\gamma_0\le\gamma_3$;
\item[{\rm (iii)}] $\displaystyle\lim_{\gamma\rightarrow\infty}b(a,\gamma)=\infty$.
\end{itemize}
\end{lemma}
\begin{proof}Our strategy is inspired by \cite{simil}, but here we consider the value of $f$ at $\eta=0$ as a function of both $a$ and $\gamma$, and study the dependence of $b(a,\gamma)$ on $\gamma$.

(i) Denote $f_{\gamma_i}=f(\eta;a,\gamma_i)$. Let $f_{\gamma_1}$ and $f_{\gamma_2}$ be positive solutions satisfying (\ref{k}), (\ref{kb1}) corresponding to $\gamma=\gamma_1,\gamma=\gamma_2$. Suppose $b(a,\gamma)$ is not strictly monotonically increasing in $\gamma$. Then it is possible to find $\gamma_1>\gamma_2$ such that $b(a,\gamma_1)\le b(a,\gamma_2)$ and $\eta_0\in[0,a)$ such that $f_{\gamma_1}(\eta_0)=f_{\gamma_2}(\eta_0)$ and $f_{\gamma_1}>f_{\gamma_2}$ on $(\eta_0,a)$, we denote  $\bar f:=f_{\gamma_1}(\eta_0)=f_{\gamma_2}(\eta_0)$.

Integrating the equation (\ref{k}) for $f_{\gamma_1}$ and $f_{\gamma_2}$ from $\eta_0$ to $a$ and obtain,
\begin{align}
&\frac{1}{2}\eta_0\bar f+\frac{1}{2}\int_{\eta_0}^{a}f_{\gamma_1}(s){\rm d}s=-\gamma_1-\phi'(\bar f)f_{\gamma_1}'(\eta_0),\label{bg1}\\
&\frac{1}{2}\eta_0\bar f+\frac{1}{2}\int_{\eta_0}^{a}f_{\gamma_2}(s){\rm d}s=-\gamma_2-\phi'(\bar f)f_{\gamma_2}'(\eta_0).\label{bg2}
\end{align}
Subtract (\ref{bg2}) from (\ref{bg1}) gives
\begin{align*}
\frac{1}{2}\int_{\eta_0}^{a}(f_{\gamma_1}(s)-f_{\gamma_2}(s)){\rm d}s=(\gamma_2-\gamma_1)+\phi'(\bar f)[f_{\gamma_2}'(\eta_0)-f_{\gamma_1}'(\eta_0)].
\end{align*}
Since $f_{\gamma_1}>f_{\gamma_2}$ on $(\eta_0,a)$, the left-hand side is positive. The right-hand side is negative because $f_{a_2}'(\eta_0)\le f_{a_1}'(\eta_0)$ at $\eta_0$ and $\gamma_2<\gamma_1$. We therefore have a contradiction, and the function $b(a,\gamma)$ must be strictly monotonically increasing in $\gamma$.

(ii) Let $0<\gamma_0\le\gamma_1<\gamma_2\le\gamma_3$. Recall the function $\t(f)$ from Lemma \ref{lkb1} and set $\t(f)=\t(f;\gamma_i)=\t_i$, where $i=1,2$. Then
\begin{align*}
|\t(f;\gamma_1)-\t(f;\gamma_2)|
=2a^{-2}\left|\int_0^f\frac{\phi'(\theta)\left\{\int_0^\theta[\t_1(s)-\t_2(s)]{\rm d}s+\frac{2\gamma_2}{a}-\frac{2\gamma_1}{a}\right\}}{\left\{\int_0^\theta [1-\tau_1(s)]{\rm d}s+\frac{2\gamma_1}{a}\right\}\left\{\int_0^\theta [1-\tau_2(s)]{\rm d}s+\frac{2\gamma_2}{a}\right\}}{\rm d}\theta\right|.
\end{align*}
Consider the function
\begin{align*}
L(\theta;\gamma)=\left(\theta+\frac{2\gamma}{a}\right)^{-1}\left\{\int_0^\theta [1-\tau(s;\gamma)]{\rm d}s+\frac{2\gamma}{a}\right\},\quad 0<\theta\le b(a,\gamma).
\end{align*}
$L(\theta;\gamma)$ is a monotonically decreasing function of $\theta$ and $L\rightarrow 1$ as $\theta\rightarrow0$. Therefore, when $0<\theta\le b(a,\gamma)$
\begin{align*}
L[b(a,\gamma);\gamma]\le L(\theta;\gamma)\le 1.
\end{align*}

We can now  write
\begin{align*}
|\t(f;\gamma_1)-\t(f;\gamma_2)|\le A(\gamma_2-\gamma_1)+B\int_0^f\frac{\phi'(\theta)}{\theta+\frac{2\gamma_1}{a}}\displaystyle\max_{0\le s\le\theta}|\t(s;\gamma_1)-\t(s;\gamma_2)|{\rm d}\theta,
\end{align*}
where
\begin{align*}
&A=16a^{-1}\gamma_0^{-2}\phi(b(a,\gamma_3))\left\{L[b(a,\gamma_1);\gamma_1]\right\}^{-1}\left\{L[b(a,\gamma_2);\gamma_2]\right\}^{-1},\\
&B=2a^{-2}\left\{L[b(a,\gamma_1);a]\right\}^{-1}\left\{L[b(a,\gamma_2);\gamma_2]\right\}^{-1},
\end{align*}
and if we set $\omega(f)=\displaystyle\max_{0\le\theta\le f}|\t(\theta;\gamma_1)-\t(\theta;\gamma_2)|$, then
\begin{align*}
\omega(f)\le A(\gamma_2-\gamma_1)+B\int_0^f\frac{\phi'(\theta)}{\theta+\frac{2\gamma_1}{a}}\omega(\theta){\rm d}\theta.
\end{align*}

Define the function
\begin{align*}
M(\gamma)=L[b(a,\gamma);\gamma]=[ab(a,\gamma)+2\gamma]^{-1}\left[\int_0^af(\eta;\gamma){\rm d}\eta+2\gamma\right].
\end{align*}
It was shown in (ii) that, since $\gamma_i\ge \gamma_0 (i=1,2)$,
\begin{align*}
f(\eta;\gamma_i)\ge f(\eta;\gamma_0)\quad {\rm on}\ [0,a).
\end{align*}
Since $f(\eta;\gamma_i)>0$ it follows that
\begin{align*}
M(\gamma_i)\ge[ab(a,\gamma_i)+2\gamma_i]^{-1}\left[\int_0^{a}f(\eta;\gamma_0){\rm d}\eta+2\gamma_0\right].
\end{align*}
Moreover, $\gamma_i\le \gamma_3$ and hence, in view of (ii), $b(a,\gamma_i)<b(a,\gamma_3)$. Therefore
\begin{align*}
M(\gamma_i)\ge[ab(a,\gamma_3)+2\gamma_3]^{-1}\left[\int_0^{a}f(\eta;\gamma_0){\rm d}\eta+2\gamma_0\right].
\end{align*}
Thus it can be seen that the constants $A$ and $B$ are uniformly bounded for $\gamma\in [\gamma_0,\gamma_3]$.

It now follows from Gronwall's Lemma (see \cite[p.24]{ode}) and the fact that $f\le b(a,\gamma_3)$, that $\t(f;\gamma)$ satisfies a Lipschitz condition in $\gamma$ which is uniform with respect to $f\in[0,b(a,\gamma_3)]$ and $\gamma\in[\gamma_0,\gamma_3]$.

From this, and the observation that $\t$ is continuously differentiable on $(0,1]$ with
\begin{align*}
\frac{\partial\t}{\partial f}=2a^{-2}\frac{\phi'(f)}{f+\frac{2\gamma}{a}}[L(f;\gamma)]^{-1}\ge 2a^{-2}\frac{\phi'(f)}{f+\frac{2\gamma}{a}},
\end{align*}
we can write

\begin{align*}
|\t(b(a,\gamma_1);\gamma_2)-\t(b(a,\gamma_2);\gamma_2)|=\int_{b(a,\gamma_1)}^{b(a,\gamma_2)}\frac{\partial\t}{\partial f}(f,\gamma_2){\rm d}f
\ge2a^{-2}\frac{\phi'(f^*)}{f^*+\frac{2\gamma_2}{a}}[b(a,\gamma_2)-b(a,\gamma_1)],
\end{align*}
by the Mean Value Theorem, for some $f^*\in(b(a,\gamma_1),b(a,\gamma_2))$. Now we consider
\begin{align*}
&|\t(b(a,\gamma_1);\gamma_2)-\t(b(a,\gamma_2);\gamma_2)|-|\t(b(a,\gamma_1);\gamma_1)-\t(b(a,\gamma_1);\gamma_2)|\\ \le&|\t(b(a,\gamma_1);\gamma_1)-\t(b(a,\gamma_2);\gamma_2)|=0.
\end{align*}
Then we have
\begin{align*}
|\t(b(a,\gamma_1);\gamma_2)-\t(b(a,\gamma_2);\gamma_2)|\le|\t(b(a,\gamma_1);\gamma_1)-\t(b(a,\gamma_1);\gamma_2)|\le K|\gamma_1-\gamma_2|,
\end{align*}
since $\t$ is Lipschitz continuous. Therefore
\begin{align*}
2a^{-2}\frac{\phi'(f^*)}{f^*+\frac{2\gamma_2}{a}}[b(a,\gamma_2)-b(a,\gamma_1)]\le K|\gamma_1-\gamma_2|.
\end{align*}

We may conclude that the function $b(a,\gamma)$ Lipschitz continuous in $\gamma$ and the Lipschitz constant is uniform in $\gamma\in[\gamma_0,\gamma_3]$.

(iii) Integrating (\ref{k}) from $\eta$ to $a$ yields
\begin{align*}
-\phi'(f(\eta))f'(\eta)=\gamma-\frac{1}{2}\int_\eta^asf'(s){\rm d}s\ge\gamma.
\end{align*}
Then we integrate from $\eta$ to $a$ and obtain
\begin{align*}
\int_{0}^{f(\eta)}\phi'(f){\rm d}f\ge\gamma(a-\eta),
\end{align*}
letting $\eta\rightarrow 0$ gives
\begin{align*}
\int_{0}^{b(a,\gamma)}\frac{\phi'(f)}{f}{\rm d}f\ge a\gamma.
\end{align*}
As $\phi'(f)$ is continuous on $[0,\infty)$ and $\phi'(0)=0$ then we have by condition (\ref{phi1}) that $b(a,\gamma)\rightarrow\infty$ as $\gamma\rightarrow\infty$.
\end{proof}

\begin{lemma}\label{lba}
$b(a,\gamma)$ has the following properties with fixed $\gamma$:
\begin{itemize}\item[{\rm (i)}] $b(a,\gamma)$ is strictly monotonically increasing in $a$;
\item[{\rm (ii)}] $\displaystyle\lim_{a\rightarrow0}b(a,\gamma)=0$;
\item[{\rm (iii)}] $b(a,\gamma)$ is Lipschitz continuous in $a$ and the Lipschitz constant is uniform in $a\in(a_0,a_3)$ and $\gamma\in(\gamma_0,\gamma_3)$, where $0\le a_0\le a_3$, $0\le \gamma_0\le \gamma_3$;
\item[{\rm (iv)}] $\displaystyle\lim_{a\rightarrow\infty}b(a,\gamma)=\infty$.\end{itemize}
\end{lemma}
\begin{proof}The proof of (i) follows from a similar argument to that of Lemma \ref{lbg} (i).

(ii) Let $a<1$, denote $N=\phi'(b(1,\gamma))$, we have $\phi'(f)\le N$ by (i). Then we get directly from (\ref{k}) that
\begin{align*}
[\phi(f(\eta))]+\frac{\eta}{2N}[\phi(f(\eta))]'\ge''\quad {\rm for}\ 0<\eta<a,
\end{align*}
multiplying $e^{\frac{\eta^2}{4N}}$ and integrating from $0$ to $\eta$ then yields
\begin{align*}
[\phi(f(\eta))]'\ge Ae^{\frac{-\eta^2}{4N}},
\end{align*}
where $A=\phi'(b(a,\gamma))f'(0)<0$. Integrating from $\eta$ to $a$ we get
\begin{align*}
\phi(f(\eta))\le-A\int_\eta^ae^{\frac{-s^2}{4N}}{\rm d}s.
\end{align*}

Now we integrate the equation (\ref{k}) from $0$ to $a$ and obtain
\begin{align*}
\frac{1}{2}\int_0^af(s){\rm d}s=\gamma-\phi'(b(a,\gamma))f'(0).
\end{align*}
Then $-A=\gamma+\displaystyle\frac{1}{2}\int_0^af(s){\rm d}s$ and we have $-A\rightarrow\gamma$ is bounded as $a\rightarrow 0$.

Therefore $\displaystyle\lim_{\eta\rightarrow 0}\phi(f(\eta))\le -A\int_0^ae^{\frac{-s^2}{4N}}{\rm d}s\rightarrow 0$ as $a\rightarrow 0$ since $e^{\frac{-s^2}{4N}}$ is bounded, which implies that $\displaystyle\lim_{a\rightarrow 0}\lim_{\eta\rightarrow0}f(\eta)=\lim_{a\rightarrow0}b(a,\gamma)=0$.

(iii) The proof is similar to that of Lemma \ref{lbg} (ii). We omit most of the details and only note the key differences. Let $0<a_0\le a_1<a_2\le a_3$. Recall the function $\t(f)$ from Lemma \ref{lkb1} and set $\tau(f)=\tau(f;a_i)=\tau_i$, where $i=1,2$. Then
\begin{align*}
|\t(f;a_1)-\t(f;a_2)|
\le&2a_2^{-2}\int_0^f\frac{\phi'(\theta)\int_0^\theta|\t_1(s)-\t_2(s)|{\rm d}s}{\left\{\int_0^\theta[1-\t_1(s)]{\rm d}s+\frac{2\gamma}{a_1}\right\}\left\{\int_0^\theta[1-\t_2(s)]{\rm d}s+\frac{2\gamma}{a_2}\right\}}{\rm d}\theta\\&+\frac{2(a_2^{2}-a_1^{2})}{a_2^2a_1^2}\int_0^f\frac{\phi'(\theta)}{\int_0^\theta[1-\t_1(s)]{\rm d}s+\frac{2\gamma}{a_1}}{\rm d}\theta.
\end{align*}
Let $0<\gamma_0\le \gamma_1<\gamma_2\le \gamma_3$ and consider the function
\begin{align*}
L(\theta;a,\gamma)=\left(\theta+\frac{2\gamma}{a}\right)^{-1}\left\{\int_0^\theta [1-\tau(s;a)]{\rm d}s+\frac{2\gamma}{a}\right\},\quad 0<\theta\le b(a,\gamma).
\end{align*}
The function $L(\theta;a)$ is clearly a monotonically decreasing function of $\theta$, and $L\rightarrow 1$ as $\theta\rightarrow0$.
Therefore
\begin{align*}
L[b(a,\gamma);a]\le L(\theta;a)\le 1 \quad {\rm for}\ 0<\theta\le b(a,\gamma).
\end{align*}

It then follows that
\begin{align*}
|\t(f;a_1)-\t(f;a_2)|\le A(a_2-a_1)+B\int_0^f\frac{\phi'(\theta)}{\theta+\frac{2\gamma}{a}}\displaystyle\max_{0\le s\le\theta}|\t(s;a_1)-\t(s;a_2)|{\rm d}\theta,
\end{align*}
where
\begin{align*}
A&=2\frac{a_1+a_2}{a_1^2a_2^2}\int_0^{b(a_1,\gamma)}\frac{\phi'(\theta)}{\theta+\frac{2\gamma_0}{a_1}}{\rm d}\theta\left\{L[b(a_1,\gamma);a_1]\right\}^{-1}\\
&\le2\frac{a_1+a_2}{a_1^2a_2^2}\int_0^{b(a_1,\gamma_3)}\frac{\phi'(\theta)}{\theta+\frac{2\gamma_0}{a_1}}{\rm d}\theta\left\{L[b(a_1,\gamma);a_1]\right\}^{-1},\\
B&=2a_2^{-2}\left\{L[b(a_1,\gamma);a_1]\right\}^{-1}\left\{L[b(a_2,\gamma);a_2]\right\}^{-1}.
\end{align*}
If we set $\omega(f)=\displaystyle\max_{0\le\theta\le f}|\t(\theta;a_1)-\t(\theta;a_2)|$, we then have
\begin{align*}
\omega(f)\le A(a_2-a_1)+B\int_0^f\frac{\phi'(\theta)}{\theta+\frac{2\gamma}{a}}\omega(\theta){\rm d}\theta.
\end{align*}

Define the function
\begin{align*}
M(a,\gamma):=L[b(a,\gamma);a]=[ab(a,\gamma)+2\gamma]^{-1}\left[\int_0^af(\eta;a,\gamma){\rm d}\eta+2\gamma\right].
\end{align*}
By the proof of (i) and Lemma \ref{lbg} (i), we have
\begin{align*}
M(a_i,\gamma_i)\ge[a_3b(a_3,\gamma_3)+2\gamma_3]^{-1}\left[\int_0^{a_0}f(\eta;a_0,\gamma_0){\rm d}\eta+2\gamma_0\right].
\end{align*}
Thus it can be seen that the constant $A$ and $B$ are uniformly bounded on the interval $[a_0,a_3]$ and $[\gamma_0,\gamma_3]$.

It now follows from Gronwall's Lemma \cite[p24]{ode} and the fact $f\le b(a_3,\gamma)$ that $\t(f;a)$ satisfies a Lipschitz condition in $a$ which is uniform with respect to $f\in[0,b(a_3,\gamma)]$ and $a\in[a_0,a_3]$.

We may conclude by a similar argument to that in proof of Lemma \ref{lbg} (ii) that the function $b(a,\gamma)$ is Lipschitz continuous in $a$ and the Lipschitz constant is uniform in $a\in(a_0,a_3)$ and $\gamma\in(\gamma_0,\gamma_3)$.

It can be shown that the Lipschitz constant is uniform in both $a\in(a_0,a_3)$ and $\gamma\in(\gamma_0,\gamma_3)$ since we proved the monotonicity on $\gamma$ of $b(a,\gamma)$ on Lemma \ref{lbg}. This result will be used in proving $b(a,\gamma)$ is a continuous function of both $a$ and $\gamma$.

(iv) Integrating (\ref{k}) from $\eta$ to $a$ yields
\begin{align*}
-\phi'(f(\eta))f'(\eta)=\gamma-\frac{1}{2}\int_\eta^asf'(s){\rm d}s\ge\gamma+\frac{\eta}{2}f(\eta)\ge\frac{\eta}{2}f(\eta).
\end{align*}
For any $a_4$ with $\eta<a_4<a$ we obtain
\begin{align*}
\int_{f(a_4)}^{f(\eta)}\frac{\phi'(f)}{f}{\rm d}f\ge\frac{1}{4}(a_4^2-\eta^2),
\end{align*}
letting $a_4\rightarrow a$ and $\eta\rightarrow 0$,
\begin{align*}
\int_{0}^{b(a,\gamma)}\frac{\phi'(f)}{f}{\rm d}f\ge\frac{1}{4}a^2.
\end{align*}
As $\phi'(f)$ is continuous on $[0,\infty)$ and $\phi'(0)=0$ then we have by condition (\ref{phi1}) that $b(a,\gamma)\rightarrow\infty$ as $a\rightarrow\infty$.\end{proof}

Now we prove that $b(a,\gamma)$ is a continuous function of both $a$ and $\gamma$ by using Lemma \ref{lbg} (ii) and Lemma \ref{lba} (iii).

\begin{lemma}
$b(a,\gamma)$ is a continuous function of $\gamma$ and $a$.
\label{lc1}
\end{lemma}
\begin{proof}Consider
\begin{align*}
|b(a,\gamma)-b(a_0,\gamma_0)|\le|b(a,\gamma)-b(a_0,\gamma)|+|b(a_0,\gamma)-b(a_0,\gamma_0)|.
\end{align*}
It was shown in the proof of Lemma \ref{lbg} (ii) that $b(a,\gamma)$ is uniformly continuous in $\gamma\in[\gamma_0,\gamma_3]$, so there exists $\mu_1$ such that $|b(a_0,\gamma)-b(a_0,\gamma_0)|<\frac{\delta}{2}$ if $|\gamma-\gamma_0|<\mu_1$. And by the proof of Lemma \ref{lba} (iii),  there exists $\mu_2$ such that $|b(a,\gamma)-b(a_0,\gamma)|<\frac{\delta}{2}$ if $|a-a_0|<\mu_2$ and $\gamma\in[\gamma_0,\gamma_3]$. Therefore $|b(a,\gamma)-b(a_0,\gamma_0)|<\frac{\delta}{2}+\frac{\delta}{2}=\delta$ if $|a-a_1|+|\gamma-\gamma_0|\le\displaystyle\max\{\mu_1,\mu_2\}$.
\end{proof}

By similar arguments to those in Lemmas \ref{lbg}, \ref{lba} and \ref{lc1} and the fact that the particular choice of $\eta=0$ in $b(a,\gamma)=f(0;a,\gamma)$ plays no special role, letting $\eta_0\in[0,\infty]$ play the same role as $0$, we can obtain the following corollary.
\begin{corollary}
For each fixed $\eta_0\in(0,a)$, if $f$ satisfies (\ref{k}) and (\ref{kb1}), then $f(\eta_0;a,\gamma)$ is a continuous function of $a$ and $\gamma$ and is monotonically increasing in both $a$ and $\gamma$.
\label{ccf}
\end{corollary}

\subsection{Solution in right-neighbourhood of $\eta=a$}
Now we consider $f$ that satisfies the equation
\begin{align}
-\frac{1}{2}\eta f'(\eta)=\varepsilon[\phi'(-f(\eta))f'(\eta)]',\quad  \eta>a.
\label{k2}
\end{align}
At the boundaries we require
\begin{align}
&\displaystyle\lim_{\eta\rightarrow\infty}f(\eta)=-V_0,\label{kkb2}\\
&\displaystyle\lim_{\eta\searrow a}f(\eta)=0,\quad \displaystyle\lim_{\eta\searrow a}\varepsilon\phi'(-f(\eta))f'(\eta)=-\gamma.\label{kkb1}
\end{align}

Next, we use the similar arguments to that of left-neighbourhood to prove the existence of a negative solution of (\ref{k2}) in a right-neighbourhood of $\eta=a$, which satisfies the boundary conditions (\ref{kkb1}).

\begin{lemma}
For given $a>0$, there exists $\delta>0$ such that in $(a,a+\delta)$ equation (\ref{k2}) has a unique solution which is negative and satisfies the boundary condition (\ref{kkb1}).
\label{lkb2}
\end{lemma}
\begin{proof}It is convenient to start by supposing that such a solution exists in a right-neighbourhood of $\eta=a$. Similarly to the proof of Lemma \ref{lkb1}, let $\sigma(f)$ be a solution of this integro-differential equation which satisfies the initial condition $\sigma(0)=a$ and is defined and continuous on an interval $[f_0,0]$ for some $f_0<0$, and is continuously differentiable on $(f_0,0)$, such that
\begin{align}
\sigma(f)=a-2\int^0_f\frac{\varepsilon\phi'(-\theta)}{\int^0_\theta \sigma(s){\rm d}s-2\gamma}{\rm d}\theta,\label{baa5}
\end{align}
if we set
\begin{align*}
\tau(f)=\frac{a}{\sigma(f)}=\frac{a}{\eta},
\end{align*}
then (\ref{baa5}) becomes
\begin{align}
\tau(f)=\frac{1}{1-2a^{-2}\int_f^0\frac{\varepsilon\phi'(-\theta)}{\int_\theta^0 \frac{1}{\tau(s)}{\rm d}s-\frac{2\gamma}{a}}}{\rm d}\theta.\label{baa6}
\end{align}

If the solution of (\ref{baa6}) is unique, the corresponding solution of equation (\ref{k2}) is also unique.

\begin{lemma}
There exists a $\mu>0$ such that (\ref{baa6}) has a unique continuous solution in $-\mu\le f\le0$, which is such that $\t(0)=1$ and $\t(f)<1$ if $-\mu\le f<0$.\label{lkb3}
\end{lemma}
\begin{proof}With $\mu$ to be chosen later, we denoted by $X$ the set of continuous functions $\t(f)$ defined in $[-\mu,0]$, satisfying $\frac{1}{2}\le\t(f)\le1$. We denote by $\Vert\centerdot\Vert$ the supremum norm on $X$. Then $X$ is a complete metric space. On $X$ we introduce the map
\begin{align*}
M(\t)(f)=\frac{1}{1-2a^{-2}\int_f^0\frac{\varepsilon\phi'(-\theta)}{\int_\theta^0 \frac{1}{\tau(s)}{\rm d}s-\frac{2\gamma}{a}}{\rm d}\theta}
\ge\frac{1}{1+2a^{-2}\int_{-\mu}^0\frac{\varepsilon\phi'(-\theta)}{\theta+\frac{2\gamma}{a}}{\rm d}\theta}.
\end{align*}
It is clear that $M(\t)(f)\le1$ is well-defined, continuous. Moreover, $M(\t)\ge\frac{1}{2}$ if
\begin{align*}
\frac{1}{1+2a^{-2}\int_{-\mu}^0\frac{\varepsilon\phi'(-\theta)}{\theta+\frac{2\gamma}{a}}{\rm d}\theta}\ge\frac{1}{2},
\end{align*}
which gives
\begin{align}
2a^{-2}\int_{-\mu}^0\frac{\varepsilon\phi'(-\theta)}{\theta+\frac{2\gamma}{a}}{\rm d}\theta\le 1.\label{baa7}
\end{align}
Therefore, if $\mu$ is chosen so small that (\ref{baa7}) is satisfied, $M$ maps $X$ into itself.

We also wish to ensure that $M$ is a contraction map. Let $\t_1,\t_2\in X$ and choose $\mu\le\frac{\gamma}{a}$, we have
\begin{align*}
\Vert M(\t_1)-M(\t_2)\Vert=&2a^{-2}\left\Vert\frac{\int_f^0\varepsilon\phi'(-\theta)\frac{\int_\theta^0\frac{1}{\t_2(s)}-\frac{1}{\t_1(s)}{\rm d}s}{(\int_\theta^0\frac{1}{\t_1(s)}{\rm d}s-\frac{2\gamma}{a})(\int_\theta^0\frac{1}{\t_2(s)}{\rm d}s-\frac{2\gamma}{a})}}{(1-2a^{-2}\int_f^0\frac{\varepsilon\phi'(-\theta)}{\int_\theta^0\frac{1}{\t_1(s)}{\rm d}s-\frac{2\gamma}{a}})(1-2a^{-2}\int_f^0\frac{\varepsilon\phi'(-\theta)}{\int_\theta^0\frac{1}{\t_2(s)}{\rm d}s-\frac{2\gamma}{a}})}\right\Vert\\ \le&2a^{-2}\left\Vert\int_{-\mu}^0\varepsilon\phi'(-\theta)\frac{\int_\theta^0\frac{\t_1(s)-\t_2(s)}{\t_1(s)\t_2(s)}{\rm d}s}{(\theta+\frac{2\gamma}{a})^2}{\rm d}\theta\right\Vert\\
\le&8a^{-2}\int_{-\mu}^0\frac{\varepsilon\phi'(-\theta)}{\theta+\frac{2\gamma}{a}}{\rm d}\theta\Vert\t_1-\t_2\Vert.
\end{align*}
It follows that $M$ is a contraction map if
\begin{align*}
8a^{-2}\int_{-\mu}^0\frac{\varepsilon\phi'(-\theta)}{\theta+\frac{2\gamma}{a}}{\rm d}\theta<1.
\end{align*}
This constitutes our third restriction on $\mu$, it clearly implies the first one. The result now follows from a standard fixed-point principle \cite{line}.\end{proof}
This concludes the proof of Lemma \ref{lkb2}.\end{proof}

For any $a>0$, the unique negative solution $f(\eta)$ defined in a right-neighbourhood of $\eta=a$, which satisfies the boundary conditions (\ref{kkb1}), may be uniquely continued forward as a function of $\eta$.  We now show that the solution can be continued forward to $\eta\rightarrow\infty$.

\begin{lemma}
For given $a,\gamma$, the unique local solution in Lemma \ref{lkb2} can be continued forward to $\eta\rightarrow\infty$.
\end{lemma}
\begin{proof}We have from (\ref{k2}) that
\begin{align}
-f'(\eta)\le\frac{2\varepsilon}{a}[\phi'(-f)f']'.\label{ns1}
\end{align}
Integrating (\ref{ns1}) from $2a$ to $\eta$ yields
\begin{align*}
-\varepsilon\phi'(-f(2a))f'(2a)-\frac{a}{2}f(2a)\ge-\varepsilon\phi'(-f(\eta))f'(\eta)-\frac{a}{2}f(\eta),
\end{align*}
then we know that $-\varepsilon\phi'(-f(\eta))f'(\eta)-\frac{a}{2}f(\eta)$ is bounded above by some positive constant $C$, so $f$ is bounded. The boundedness of $-f'(\eta)$ for $\eta>a+\frac{\delta}{2}$ follows similarly to (\ref{fpr1}) for $\delta>0$, so it follows from \cite[Theorem 1.186]{odea} that the solution of (\ref{k2}) can be continuous forward to $\eta\rightarrow\infty$.
\end{proof}

\subsection{Properties of $d(a,\gamma)$}
 Now define
\begin{align*}
 d(a,\gamma)=\displaystyle\lim_{\eta\rightarrow\infty}f(\eta;a,\gamma).
 \end{align*}
  The following discussions on $d(a,\gamma)$ are used in proving existence of self-similar solution by shooting from $\eta=a$ with $\gamma$, the derivatives of $\phi(f)$ at $\eta=a$, to $\displaystyle\lim_{\eta\rightarrow\infty}f(\eta;a,\gamma)$. The strategy in studying the properties of $d(a,\gamma)$ is similar to that used to prove the properties of $b(a,\gamma)$, but some arguments are more involved since $(a,\infty)$ is unbounded.

\begin{lemma}\label{ldg}
$d(a,\gamma)$ has the following properties with fixed $a$:
\begin{itemize}
\item[{\rm (i)}] $d(a,\gamma)$ is strictly monotonically decreasing in $\gamma$$;$
\item[{\rm (ii)}] $\displaystyle\lim_{\gamma\rightarrow\infty}d(a,\gamma)=-\infty$;
\item[{\rm (iii)}] $\displaystyle\lim_{\gamma\rightarrow0}d(a,\gamma)=0$.
\end{itemize}
\end{lemma}
\begin{proof}The proof of (i) follows from the similar argument to that of Lemma \ref{lbg} (i).

(ii) Suppose $d(a,\gamma)$ does not satisfy $\displaystyle\lim_{\gamma\rightarrow\infty}d(a,\gamma)=-\infty$. Then there exists $M>0$ such that $d(a,\gamma)\ge-M$ for all $\gamma$, which implies $|f(\eta_0)|\le M$ for each fixed $\eta_0>a$.

Integrating (\ref{k2}) from $a$ to $\eta_0$ gives
\begin{align*}
\gamma=-\frac{1}{2}\eta_0f(\eta_0)+\frac{1}{2}\int_a^{\eta_0}f(s){\rm d}s-\varepsilon\phi'(-f(\eta_0))f'(\eta_0).
\end{align*}
Since $\displaystyle\int_a^{\eta_0}f(s){\rm d}s$ is negative, using the upper bound of $|f|$, we get
\begin{align*}
-\varepsilon\phi'(M)f'(\eta_0)>\gamma-\frac{M\eta_0}{2}-\frac{1}{2}\int_a^{\eta_0}f(s){\rm d}s>\gamma-\frac{M\eta_0}{2}.
\end{align*}
By (\ref{fpr1}) we know that for all $\eta_0>a+1$
\begin{align*}
\frac{4\gamma\phi'(M)}{2\eta_0-1}>-\phi'(M)f'(\eta_0)>\gamma-\frac{M\eta_0}{2}.
\end{align*}
If we rewrite as $\displaystyle\frac{M\eta_0}{2}>\gamma\left(1-\frac{4\phi'(M)}{2\eta_0-1}\right)$ and choosing and fixing $\eta_0$ sufficient large such that $1-\displaystyle\frac{4\phi'(M)}{2\eta_0-1}>\frac{1}{2}$, we then have $M\eta_0>\gamma$ for all $\gamma$. But this is a contradiction, so if $\gamma\rightarrow\infty$, we have $d(a,\gamma)\rightarrow-\infty$.

(iii) Integrating (\ref{k2}) from $a$ to $\eta$ and letting $\eta\rightarrow\infty$, together with $\displaystyle\lim_{\eta\rightarrow\infty}f'(\eta)=0$ by (\ref{f0}), we get
\begin{align*}
-\frac{a}{2}\int_a^\infty f'(s){\rm d}s=-\frac{a}{2}d(a,\gamma)\le\gamma.
\end{align*}

Then the result follows from
\begin{align*}
-\frac{a}{2}d(a,\gamma)\rightarrow0 \quad {\rm as} \ \gamma\rightarrow0.
\end{align*}
\end{proof}

\begin{lemma}\label{lda}
$d(a,\gamma)$ has the following properties with fixed $\gamma$:
\begin{itemize}
\item[{\rm (i)}] $d(a,\gamma)$ is strictly monotonically increasing in $a$$;$
\item[{\rm (ii)}]$\displaystyle\lim_{a\rightarrow\infty}d(a,\gamma)=0$.
\end{itemize}
\end{lemma}
\begin{proof}The proof of (i) follows from the similar arguments to that of Lemma \ref{lbg} (i).

(ii) Integrating (\ref{k2}) from $a$ to $\eta$ we get
\begin{align*}
\varepsilon\phi'(-f(\eta))f'(\eta)\ge-\gamma-\frac{a}{2}\int^\eta_af'(s){\rm d}s=\frac{a}{2}\left(-\frac{2\gamma}{a}-f(\eta)\right),
\end{align*}
then the result follows from the fact that $f'<0$, gives $-\frac{2\gamma}{a}-f<0$.
\end{proof}

Next, we prove $f$ is a continuous function of $a$ and $\gamma$ respectively by using an iterative method. This result will be used to prove $d(a,\gamma)$ is a continuous function of both $a$ and $\gamma$.
\begin{lemma}
For each fixed $\eta^*>a$, if $f$ satisfies (\ref{k2}) and (\ref{kkb1}), then
\begin{itemize}
\item[{\rm (i)}] $f(\eta^*;a,\gamma)$ is a continuous function of $\gamma$ for fixed $a$;
\item[{\rm (ii)}] $f(\eta^*;a,\gamma)$ is a continuous function of $a$ for fixed $\gamma$.
\end{itemize}
\label{lfc}
\end{lemma}
\begin{proof}First we prove $f(\eta^*;a,\gamma)$ is a continuous function $\gamma$.

Let $0<\gamma_0\le \gamma_1<\gamma_2\le \gamma_3$. Recall the function $\t(f)$ from the proof of Lemma \ref{lkb2} and set $\tau(f)=\tau(f;\gamma_i)=\tau_i$, where $i=1,2$. Let $\eta\in(a,\eta_0]$ satisfies $\frac{1}{2}\le\t(f)<1$ and
\begin{align}
\eta_0f(\eta_0;a,\gamma_3)-2\gamma_0<0,
\end{align}
for $f(\eta)\in[-\mu,0]$. Then
\begin{align*}
|\t(f;\gamma_1)-\t(f;\gamma_2)|=&\left|\frac{2a^{-2}\int_f^0\left(\frac{\varepsilon\phi'(-\theta)}{\int^0_\theta\frac{1}{\t(s)}{\rm d}s-\frac{2\gamma_1}{a}}-\frac{\varepsilon\phi'(-\theta)}{\int^0_\theta\frac{1}{\t(s)}{\rm d}s-\frac{2\gamma_2}{a}}\right){\rm d}\theta}{\left(1-2a^{-2}\int_f^0\frac{\varepsilon\phi'(-\theta)}{\int_\theta^0 \frac{1}{\tau(s)}{\rm d}s-\frac{2\gamma_1}{a}}{\rm d}\theta\right)\left(1-2a^{-2}\int_f^0\frac{\varepsilon\phi'(-\theta)}{\int_\theta^0 \frac{1}{\tau(s)}{\rm d}s-\frac{2\gamma_2}{a}}{\rm d}\theta\right)}\right|\\
\le &2a^{-2}\left|\int_f^0\frac{\varepsilon\phi'(-\theta)\left(\int_\theta^0\frac{1}{\t_2(s)}-\frac{1}{\t_1(s)}{\rm d}s+\frac{2}{a}(\gamma_1-\gamma_2)\right)}{\left(\int_\theta^0\frac{1}{\t_1(s)}{\rm d}s-\frac{2\gamma_1}{a}\right)\left(\int_\theta^0\frac{1}{\t_2(s)}{\rm d}s-\frac{2\gamma_2}{a}\right)}{\rm d}\theta\right|,
\end{align*}
since $\t(f)<1$ implies $1-2a^{-2}\int_f^0\frac{\varepsilon\phi'(-\theta)}{\int_\theta^0 \frac{1}{\tau(s)}{\rm d}s-\frac{2\gamma}{a}}{\rm d}\theta>1$.

Then we have
\begin{align*}
|\t(f;\gamma_1)-\t(f;\gamma_2)|\le &2a^{-2}\left|\int_f^0\frac{\varepsilon\phi'(-\theta)\left(\int_\theta^0\frac{\t_1(s)-\t_2(s)}{\t_1(s)\t_2(s)}{\rm d}s+\frac{2}{a}(\gamma_1-\gamma_2)\right)}{\left(\int_\theta^0\frac{1}{\t_1(s)}{\rm d}s-\frac{2\gamma_1}{a}\right)\left(\int_\theta^0\frac{1}{\t_2(s)}{\rm d}s-\frac{2\gamma_2}{a}\right)}{\rm d}\theta\right|\\ \le&2a^{-2}\left|\int_f^0\frac{\varepsilon\phi'(-\theta)\left(4\int_\theta^0\t_1(s)-\t_2(s){\rm d}s+\frac{2}{a}(\gamma_1-\gamma_2)\right)}{\left(\int_\theta^0\frac{1}{\t_1(s)}{\rm d}s-\frac{2\gamma_1}{a}\right)\left(\int_\theta^0\frac{1}{\t_2(s)}{\rm d}s-\frac{2\gamma_2}{a}\right)}{\rm d}\theta\right|,
\end{align*}
since $\t(f)>\frac{1}{2}$.

Consider the function
\begin{align*}
L(\theta;\gamma)=\left(\theta-\frac{2\gamma}{a}\right)^{-1}\left(\int_\theta^0\frac{1}{\t(s)}{\rm d}s-\frac{2\gamma}{a}\right)>0, \quad -\mu\le\theta\le0.
\end{align*}
$L$ is a monotonically increasing function of $\theta$ since
\begin{align*}
\frac{\partial L}{\partial \theta}=\frac{\int_\theta^0\frac{1}{\t(\theta)}-\frac{1}{\t(s)}{\rm d}s+\frac{2\gamma}{a}\left(1+\t(\theta)\right)}{\left(\theta-\frac{2\gamma}{a}\right)^2}>0,
\end{align*}
and $L\rightarrow 1$ as $\theta\rightarrow0$. Therefore $L(f(\eta_0;a,\gamma);\gamma)\le L(\theta;\gamma)<1$ when $-\mu\le\theta\le0$.

We can now write
\begin{align*}
|\t(f;\gamma_1)-\t(f;\gamma_2)|\le A(\gamma_2-\gamma_1)+B\int_f^0\frac{\varepsilon\phi'(-\theta)}{\frac{2\gamma_1}{a}-\theta}\max_{f\le s\le0}|\t(s;\gamma_1)-\t(s;\gamma_2)|{\rm d}\theta,
\end{align*}
where
\begin{align*}
&A=16a^{-1}\gamma_0^2\varepsilon\phi'(-f(\eta_0;a,\gamma_1))[L(f(\eta_0;a,\gamma_1))]^{-1}[L(f(\eta_0;a,\gamma_2))]^{-1},\\
&B=8a^{-2}[L(f(\eta_0;a,\gamma_1))]^{-1}[L(f(\eta_0;a,\gamma_2))]^{-1},
\end{align*}
and if we set $\omega(f)=\displaystyle\max_{f\le\theta\le0}|\t(\theta;\gamma_1)-\t(\theta;\gamma_2)|$, then
\begin{align*}
\omega(f)\le A(\gamma_2-\gamma_1)+B\int_f^0\frac{\varepsilon\phi'(-\theta)}{\frac{2\gamma_1}{a}-\theta}\omega(\theta){\rm d}\theta.
\end{align*}
Define the function
\begin{align*}
M(\gamma)=L(f(\eta_0;a,\gamma);\gamma) =(a\,f(\eta_0;a,\gamma);\gamma)-2\gamma)^{-1}\left(\int_a^{\eta_0}f(s;a,\gamma){\rm d}s-\eta_0f(\eta_0;a,\gamma)-2\gamma\right).
\end{align*}
It was shown in Lemma \ref{ldg} (i) that, since $\gamma_i\le\gamma_3$, $f(\eta;\gamma_i)\ge f(\eta;\gamma_3)$ on $(a,\eta_0]$. Since $f(\eta;\gamma_i)<0$, it follows that
\begin{align*}
M(\gamma_i)\ge(af(\eta_0;a,\gamma_3);\gamma_3)-2\gamma_3)^{-1}\left(\int_a^{\eta_0}f(s;a,\gamma){\rm d}s-\eta_0f(\eta_0;a,\gamma_3)-2\gamma_0\right)>0.
\end{align*}
Thus it can be seen that the constants $A$ and $B$ are uniformly bounded for $\gamma\in[\gamma_0,\gamma_3]$.

It now follows from Gronwall's Lemma (\cite[p24]{ode}) and the fact that $f(\eta;a,\gamma)\ge f(\eta_0;a,\gamma_3)$, $\t(f;\gamma)$ satisfies a Lipshichtz condition in $\gamma$ which is uniform with respect to $f\in[f(\eta_0;a,\gamma_3),0]$ and $\gamma\in[\gamma_0,\gamma_3]$.

The observation that $\t$ is continuously differentiable on $\left[\frac{1}{2},1\right)$ with
\begin{align*}
\frac{\partial \t}{\partial f}=\frac{1}{2}a^{-2}\frac{\varepsilon\phi'(-f)}{\frac{2\gamma}{a}-f}[L(f;\gamma)]^{-1}\ge\frac{1}{2}a^{-2}\frac{\varepsilon\phi'(-f)}{\frac{2\gamma}{a}-f}.
\end{align*}
Then the result follows from the similar argument to that in proof of Lemma \ref{lbg} (ii) that when the function $f(\eta;a,\gamma)$ for $\eta\in(a,\eta_0]$ is Lipschitz continuous in $\gamma$ and the Lipschitz constant is uniform in $\gamma\in[\gamma_0,\gamma_3]$, since $\eta=\eta_0$ is not special.

Now we prove $f(\eta_1;a,\gamma)$ is a continuous function when $\eta_1>\eta_0$. We consider in two cases.

\noindent{\bf Case A. }Consider $\gamma_0>\gamma$. First, since $f'$ is bounded for $\eta>a+\frac{\zeta}{2}$, we can choose a fixed $\eta_1$ such that $|f(\eta_0;a,\gamma_0)-f(\eta_1;a,\gamma_0)|<\frac{\delta}{2}$. We know there exists $\mu$ such that $|f(\eta_0;a,\gamma)-f(\eta_0;a,\gamma_0)|<\frac{\delta}{2}$ for $|\gamma-\gamma_0|<\mu$. Then
\begin{align*}
|f(\eta_0;a,\gamma)-f(\eta_1;a,\gamma_0)| \le|f(\eta_0;a,\gamma)-f(\eta_0;a,\gamma_0)|+|f(\eta_0;a,\gamma_0)-f(\eta_1;a,\gamma_0)|<\frac{\delta}{2}+\frac{\delta}{2}=\delta.
\end{align*}
Since $\gamma_0>\gamma$, then by Lemma \ref{ldg} (i) we have
\begin{align*}
f(\eta_1;a,\gamma_0)+\delta>f(\eta_0;a,\gamma)>f(\eta_1;a,\gamma)>f(\eta_1;a,\gamma_0),
\end{align*}
so that $f(\eta_1;a,\gamma)-f(\eta_1;a,\gamma_0)<\delta$ if $\gamma_0-\gamma<\mu$.

\noindent{\bf Case B. }Now consider $\gamma_0<\gamma$ and denote $f(\eta;a,\gamma)=f$ and $f(\eta;a,\gamma_0)=f_0$. We know that
\begin{align*}
-\frac{a}{2}f'(\eta)\le\varepsilon[\phi'(-f(\eta))f'(\eta)]',\quad \eta>a,
\end{align*}
letting $k>1$ and $ka<\eta_0$, integrating (\ref{k2}) from $ka$ to $\eta$ yield
\begin{align*}
-\varepsilon\phi'(-f(\eta))f'(\eta)-\frac{a}{2}f(\eta)\le-\varepsilon\phi'(-f(ka))f'(ka)-\frac{a}{2}f(ka).
\end{align*}
Letting $\eta\rightarrow\eta_1$ gives
\begin{align}
-f(\eta_1)<-\frac{2\varepsilon}{a}\phi'(-f(ka))f'(ka)-f(ka).
\label{dg4f}
\end{align}
Now consider the equation for $f_0$. Integrating from $a$ to $ka$, we get
\begin{align*}
-\frac{ka}{2}f_0(ka)+\frac{1}{2}\int_a^{ka}f_0(s){\rm d}s=\varepsilon\phi'(-f_0(ka))f_0(ka)+\gamma_0,
\end{align*}
we know that $f_0(\eta_1)<f_0(\eta)$ for $\eta\in(\eta_0,\eta_1]$, then
\begin{align}
f_0(\eta_1)<\frac{k}{k-1}f_0(ka)+\frac{2\varepsilon}{(k-1)a}\phi'(-f_0(ka))f_0'(ka)+\frac{2\gamma_0}{(k-1)a}.
\label{dg3f}
\end{align}
Combining (\ref{dg4f}) and (\ref{dg3f}) we have
\begin{align*}
f_0(\eta_1)-f(\eta_1)<f_0(ka)-f(ka)+\frac{1}{k-1}f_0(ka)+\frac{2\gamma_0}{(k-1)a}-\frac{2\varepsilon}{a}\phi'(-f(ka))f'(ka).
\end{align*}
Choosing $k$ such that $a<ka<\eta_1$ to satisfy $\displaystyle\frac{2\gamma_0}{(k-1)a}+\frac{f_0(ka)}{k-1}<\frac{\delta}{3}$. For $ka>a+\frac{\zeta}{2}$ we have $-f'(ka)<\displaystyle\frac{16\gamma}{4\zeta+\zeta^2}$ by (\ref{fd}). Now choose and fix $k$ so that $-\displaystyle\frac{2\varepsilon}{a}\phi'(-f(ka))f'(ka)<\frac{\delta}{3}$. With this $k$, we know there exists $\mu>0$ such that for $\gamma-\gamma_0<\mu$
\begin{align*}
f_0(ka)-f(ka)<\frac{\delta}{3}.
\end{align*}

Therefore $f(\eta_1;a,\gamma_0)-f(\eta_1;a,\gamma)<\delta$ if $\gamma-\gamma_0\le\mu$.

We can now conclude $f(\eta_1;a,\gamma)$ is a continuous function of $\gamma$ for fixed $a$. It can be prove iteratively that $f(\eta;a,\gamma)$ is a continuous function of $\gamma$ at fixed $\eta\in (a,\infty)$ with fixed $a$. Moreover, for fixed $\gamma$, $f(\eta;a,\gamma)$ is a continuous function of $a$ can be proved by using the similar argument.
\end{proof}

The following corollary is obtained directly from Lemma \ref{lfc}
\begin{corollary}\label{lcd}
If $f$ satisfies (\ref{k2}) and (\ref{kkb1}), then
\begin{itemize}
\item[{\rm (i)}] $d(a,\gamma)$ is a continuous function of $\gamma$ for fixed $a$;
\item[{\rm (ii)}] $d(a,\gamma)$ is a continuous function of $a$ for fixed $\gamma$.
\end{itemize}
\end{corollary}

Now we prove that $d(a,\gamma)$ is a continuous function of both $a$ and $\gamma$ by using Lemma \ref{ldg} (ii) and Lemma \ref{lda} (ii).
\begin{lemma}
$d(a,\gamma)$ is a continuous function of $a$ and $\gamma$.
\label{lcag}
\end{lemma}
\begin{proof}If $d(a,\gamma)$ is continuous with $a$ and $\gamma$, then for all $\delta>0$ there exists $\mu>0$ such that if $|(a,\gamma)-(a_0,\gamma_0)|<\mu$ then $|d(a,\gamma)-d(a_0,\gamma_0)|<\delta$.

\noindent{\bf Case 1.} For $a>a_0$ and $\gamma<\gamma_0$, we choose a fixed $\eta_0$ such that $|f(\eta_0;a_0,\gamma_0)-d(a_0,\gamma_0)|<\frac{\delta}{2}$. We know from Lemma \ref{lfc} that there exists $\mu$ such that $|f(\eta_0;a,\gamma)-f(\eta_0;a,\gamma_0)|<\frac{\delta}{2}$ for $|(a,\gamma)-(a,\gamma_0)|<\frac{\mu}{2}$. Then
\begin{align*}
|f(\eta_0;a,\gamma)- d(a_0,\gamma_0)| \le|f(\eta_0;a,\gamma)-f(\eta_0;a,\gamma_0)|+|f(\eta_0;a,\gamma_0)- d(a_0,\gamma_0)|<\frac{\delta}{2}+\frac{\delta}{2}=\delta.
\end{align*}
Since the sequence $a>a_0$ and $\gamma<\gamma_0$, then we have
\begin{align*}
 d(a_0,\gamma_0)+\delta>f(\eta_0;a,\gamma)>d(a,\gamma)> d(a_0,\gamma_0),
\end{align*}
then $ d(a,\gamma)- d(a_0,\gamma_0)<\delta$ as $|(a,\gamma)-(a_0,\gamma_0)|<\mu$.

\noindent{\bf Case 2.} For $a<a_0$, it follow by using the similar argument as in Lemma \ref{lfc} Case B, but considering
\begin{align*}
|d(a,\gamma)-d(a_0,\gamma_0)|\le|d(a,\gamma)-d(a_0,\gamma)|+|d(a_0,\gamma)-d(a_0,\gamma_0)|.
\end{align*}
\noindent{\bf Case 3.} For $\gamma_0<\gamma$, the result follows from the similar approach as in Lemma \ref{lfc} Case B, but considering
\begin{align*}
|d(a,\gamma)-d(a_0,\gamma_0)|\le|d(a,\gamma)-d(a,\gamma_0)|+|d(a,\gamma_0)-d(a_0,\gamma_0)|.
\end{align*}
\end{proof}

\subsection{Two-parameter shooting method}
In this section we will use two-parameter shooting to show that for each $U_0,V_0>0$, there exist $a,\gamma>0$ such that the solution $f(\eta;a,\gamma)$ of (\ref{j}) satisfies $b(a,\gamma)=U_0$ and $d(a,\gamma)=-V_0$.\\

We will use the following lemma which can be found in \cite[Lemma 2.8]{clas}.
\begin{lemma}
Suppose that $\Lambda_1$ and $\Lambda_2$ are two connected open sets of $\mathbb{R}^2$, with components (maximal connected subset) $\tilde\Lambda_1\subset\Lambda_1$ and $\tilde\Lambda_2\subset\Lambda_2$ such that $\tilde\Lambda_1\cap\tilde\Lambda_2$ is disconnected. Then $\Lambda_1\cup\Lambda_2\neq\mathbb{R}^2$.
\label{top}
\end{lemma}
This result also applies to every subset of $\mathbb{R}^2$ which is homeomorphic to the entire plane \cite[p.31]{clas}. We can apply it to the set $(0,\infty)\times(0,\infty)$, for example, if we define a homomorphism $g:(0,\infty)\times(0,\infty)\rightarrow\mathbb{R}^2$ such that $g(x,y)=(\log x,\log y)$.

\begin{theorem}\label{twe1}
Suppose $\varepsilon>0$, then there exists a unique solution $f$ of problem (\ref{j}).
\end{theorem}
\begin{proof}First we identify four ``bad" sets
\begin{align*}
&\Gamma_1=\left\{(a,\gamma)\,\big|\,b(a,\gamma)>U_0\right\},\\
&\Gamma_2=\left\{(a,\gamma)\,\big|\,b(a,\gamma)<U_0\right\},\\
&\Gamma_3=\left\{(a,\gamma)\,\big|\,d(a,\gamma)>-V_0\right\},\\
&\Gamma_4=\left\{(a,\gamma)\,\big|\,d(a,\gamma)<-V_0\right\}.
\end{align*}
We combine the $\Gamma_i$ to form two new sets, as follows:
\begin{align*}
\Lambda_1=\Gamma_1\cup\Gamma_4,\\
\Lambda_2=\Gamma_2\cup\Gamma_3.
\end{align*}
It is easy to see that if $(a,\gamma)$ is in $(0,\infty)\times(0,\infty)$ but not in $\Lambda_1\cup\Lambda_2$, then we have a solution $f(\eta;a,\gamma)$ such that $b(a,\gamma)=U_0$ and $d(a,\gamma)=-V_0$. Now we want to show that $\Lambda_1$ and $\Lambda_2$ satisfy the hypothesis of Lemma \ref{top}.

These sets are clearly open in $(0,\infty)\times(0,\infty)$ since $b(a,\gamma)$ and $d(a,\gamma)$ are continuous functions of $a$ and $\gamma$ by Lemma \ref{lc1} and \ref{lcag}. $\Gamma_1$ and $\Gamma_2$ are non-empty since $\displaystyle\lim_{a\rightarrow 0}b(a,\gamma)=0$ and $\displaystyle\lim_{a\rightarrow \infty}b(a,\gamma)=\infty$ by Lemma \ref{lba}. Moreover, $\displaystyle\lim_{a\rightarrow \infty}d(a,\gamma)=0$ and $\displaystyle\lim_{\gamma\rightarrow\infty}d(a,\gamma)=-\infty$ by Lemma \ref{ldg}, yielding $\Gamma_3$ and $\Gamma_4$ are non-empty. Therefore, $\Lambda_1$ and $\Lambda_2$ are open and non-empty.

\begin{lemma}
The sets $\Lambda_1$ and $\Lambda_2$ are connected.\label{lcon}
\end{lemma}
\begin{proof}In the following, we will exploit the monotonicity of $b(a,\gamma)$ and $d(a,\gamma)$ in $a$ and $\gamma$. First we prove that $\Gamma_1$, $\Gamma_2$, $\Gamma_3$ and $\Gamma_4$ are each connected. As an example, we prove that $\Gamma_1$ is connected. Given two points $(\tilde a,\tilde\gamma),(\hat a,\hat\gamma)\in\Lambda_1$, there are two cases:
\begin{itemize}
\item[(i).] $\tilde a>\hat a$ and $\tilde\gamma\ge\hat\gamma$;
\item[(ii).] $\tilde a\ge\hat a$ and $\tilde\gamma<\hat\gamma$.
\end{itemize}
The following figures describe an admissible step path, contained in $\Gamma_1$, that connects $(\tilde a,\tilde\gamma)$ and $(\hat a,\hat\gamma)$ in each of two cases
\begin{figure}[H]
\centering
\begin{minipage}{0.49\linewidth}
\flushleft
\includegraphics[width=1.1\textwidth]{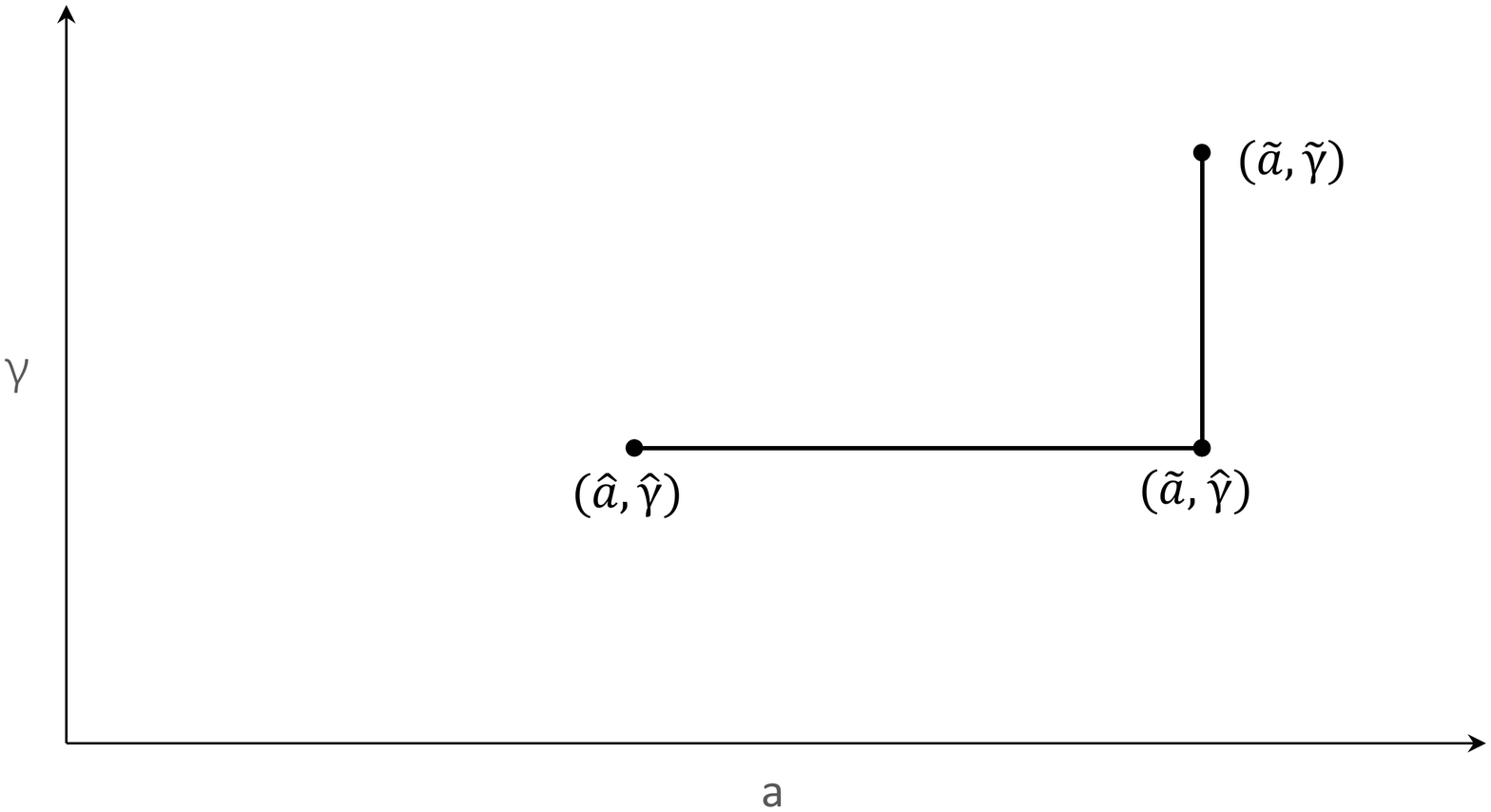}
\captionsetup{font={small}}\caption{step-path of $\Gamma_1$(i)}\label{figure1}
\end{minipage}
\begin{minipage}{0.49\linewidth}
\flushright
\includegraphics[width=1.1\textwidth]{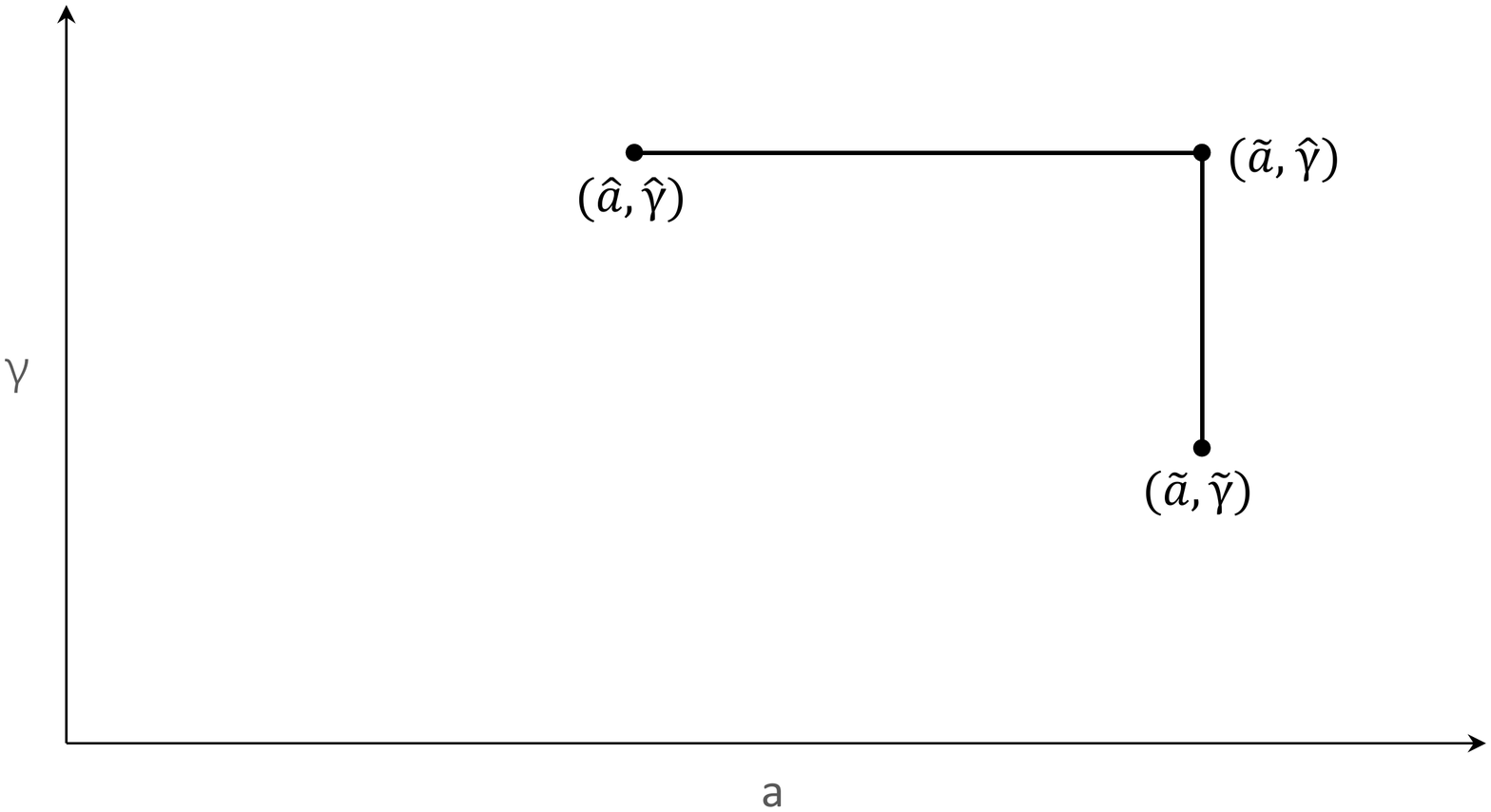}
\captionsetup{font={small}}\caption{step-path of $\Gamma_1$(ii)}\label{figure2}
\end{minipage}
\end{figure}
\noindent In Figure \ref{figure1}, if $\tilde a>\hat a$, then we have $(\tilde a,\hat\gamma)\in\Gamma_1$, since $b(\tilde a,\hat\gamma)>b(\hat a,\hat\gamma)>U_0$ by Lemma \ref{lba} (i). It follows that the path connecting $(\tilde a,\hat\gamma)$ and $(\tilde a,\tilde\gamma)$ belongs to $\Gamma_1$, since $(\tilde a,\hat\gamma),(\tilde a,\tilde\gamma)\in\Gamma_1$ and $b(a,\gamma)$ is monotonically increasing in $\gamma$ by Lemma \ref{lbg} (i). Similarly, in Figure \ref{figure2}, $(\tilde a,\hat\gamma)\in\Gamma_1$ as $\tilde a\ge\tilde a$, since $b(a,\gamma)$ is increasing in $a$ by Lemma \ref{lba} (i), then the path connecting $(\tilde a,\hat\gamma)$ and $(\tilde a,\tilde\gamma)$ belongs to $\Gamma_1$ by Lemma \ref{lba} (i). We can prove by using a similar argument that $\Gamma_2$, $\Gamma_3$ and $\Gamma_4$ are each connected.

We now prove $\Gamma_1\cap\Gamma_4$ and $\Gamma_2\cap\Gamma_3$ are non-empty.

For fixed $a>0$, since $\displaystyle\lim_{\gamma\rightarrow\infty}b(a,\gamma)=\infty$ by Lemma \ref{lbg} (iii) and $\displaystyle\lim_{\gamma\rightarrow\infty}d(a,\gamma)=-\infty$ by Lemma \ref{ldg} (ii) , we can find $\check\gamma$ large enough such that $b(a,\check\gamma)>U_0$ and $d(a,\check\gamma)<-V_0$. It follows that for $\check\gamma$ sufficiently large, $(a,\check\gamma)\in\Gamma_1\cap\Gamma_4$, so $\Gamma_1\cap\Gamma_4\ne\emptyset$. Similarly, given $\widetilde\gamma>0$, there exists $\widehat a$ small enough that $b(\widehat a,\widetilde\gamma)<U_0$ since $\displaystyle\lim_{a\rightarrow0}b(a,\gamma)=0$ by Lemma \ref{lba} (ii). Then choose $\widehat\gamma$ smaller than $\widetilde\gamma$ if necessary to ensure that $d(\widehat a,\widehat\gamma)>-V_0$ and $b(\widehat a,\widehat\gamma)<U_0$ since $\displaystyle\lim_{\gamma\rightarrow0}d(a,\gamma)=0$ by Lemma \ref{ldg} (iii) and $b(a,\gamma)$ is monotonically increasing in $\gamma$ by Lemma \ref{lbg} (i). It then follows that $(\widehat a,\widehat\gamma)\in\Gamma_2\cap\Gamma_3$, so $\Gamma_2\cap\Gamma_3\ne\emptyset$.

Then, since $\Gamma_1\cap\Gamma_4\ne\emptyset$, we can always find a point belonging to $\Gamma_1\cap\Gamma_4$ that is path connected to both $(\hat a,\hat\gamma)\in\Gamma_1$ and $(a^*,\gamma^*)\in\Gamma_4$, since $\Gamma_1$ and $\Gamma_4$ are each connected.

\begin{figure}[H]
\centering
\includegraphics[width=0.6\textwidth]{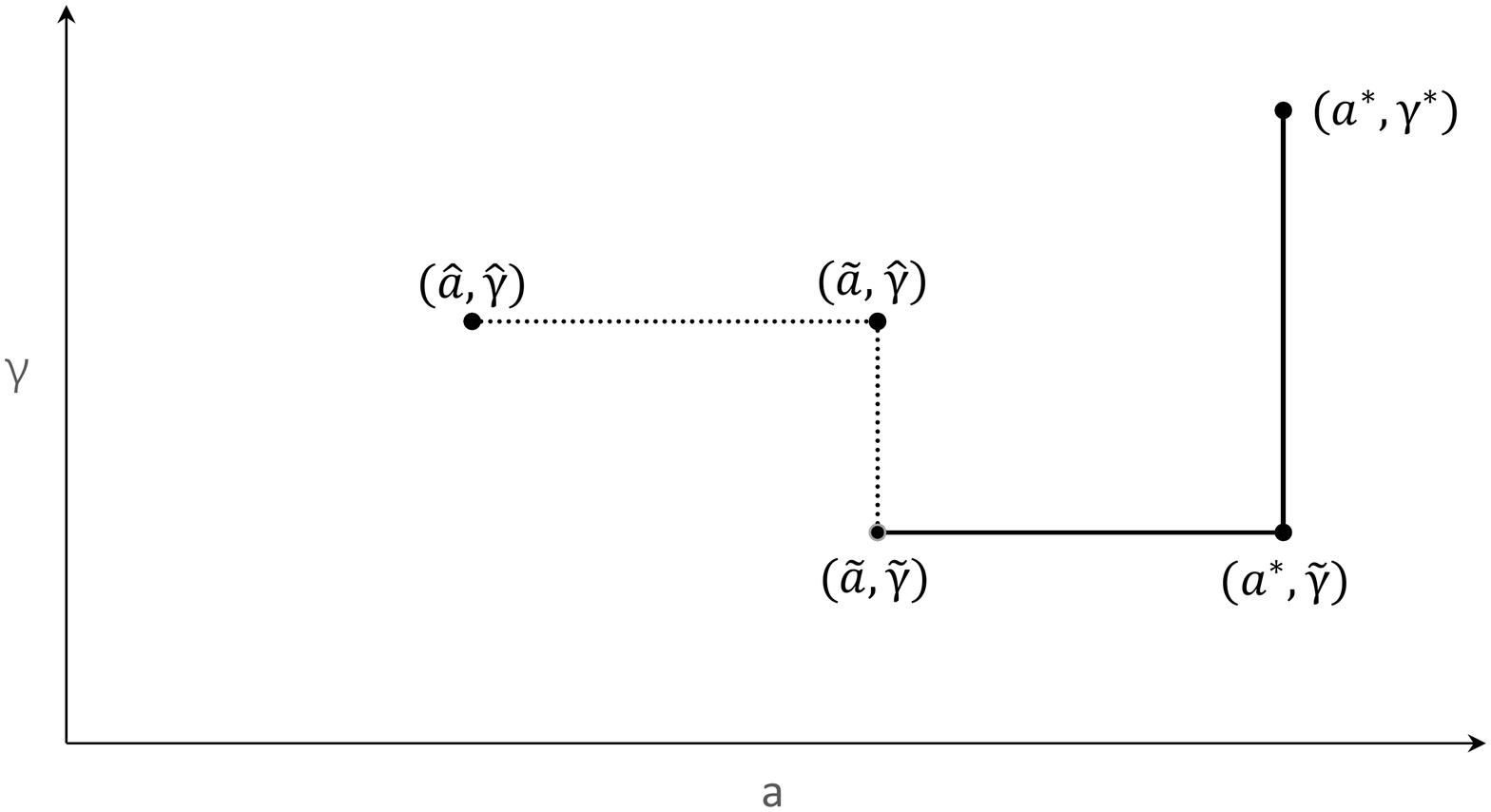}
\captionsetup{font={small}}\caption{path-connectedness of $\Lambda_1$}\label{figure3}
\end{figure}

\noindent For example, in Figure \ref{figure3}, the solid lines indicate that the path belongs to $\Gamma_1$ and the dashed lines indicate that the path belongs to $\Gamma_4$. We can find $(\tilde a,\tilde\gamma)\in \Gamma_1\cap\Gamma_4$ since $\Gamma_1\cap\Gamma_4\ne\emptyset$. If $(\hat a,\hat\gamma)\in\Gamma_4$ and $(a^*,\gamma^*)\in\Gamma_1$, then there are step paths each connecting $(\hat a,\hat\gamma)$ and $(\tilde a,\tilde\gamma)$, $(\tilde a,\tilde\gamma)$ and $(a^*,\gamma^*)$, since $\Gamma_1$ and $\Gamma_4$ are each connected.

Therefore, $\Lambda_1$ is connected, and similarly, $\Lambda_2$ is connected since $\Gamma_2\cap\Gamma_3\ne\emptyset$ and $\Gamma_2,\Gamma_3$ are each connected.\end{proof}

Now we take $\tilde\Lambda_1=\Lambda_1$, $\tilde\Lambda_2=\Lambda_2$.

Next we will show that $\Lambda_1\cap\Lambda_2$ is disconnected. We have
\begin{align*}
\Lambda_1\cap\Lambda_2=(\Gamma_1\cap\Gamma_2)\cup(\Gamma_1\cap\Gamma_3)\cup(\Gamma_2\cap\Gamma_4)\cup(\Gamma_3\cap\Gamma_4).
\end{align*}
Clearly $\Gamma_1\cap\Gamma_2$, $\Gamma_3\cap\Gamma_4$ are empty.

\begin{lemma}
$\Lambda_1\cap\Lambda_2$ is disconnected.\label{lds}
\end{lemma}
\begin{proof}For fixed $\gamma$, we can find $\tilde a$ large enough such that $b(\tilde a,\gamma)>U_0$ and $d(\tilde a,\gamma)>-V_0$, since $\displaystyle\lim_{a\rightarrow\infty}b(a,\gamma)=\infty$ by Lemma \ref{lba} (iv) and $\displaystyle\lim_{a\rightarrow\infty}d(a,\gamma)=0$ by Lemma \ref{lda} (iii). It follows that for $\tilde a$ sufficient large, $(\tilde a,\gamma)\in\Gamma_1\cap\Gamma_3$, so $\Gamma_1\cap\Gamma_3\ne\emptyset$. Similarly, given $\hat a>0$, there exits $\gamma^*$ large enough such that $d(\hat a,\gamma^*)<-V_0$, since $\displaystyle\lim_{\gamma\rightarrow\infty}d(a,\gamma)=-\infty$ by Lemma \ref{ldg} (ii). Then choose $a^*$ smaller than $\hat a$ if necessary to ensure that $d(a^*,\gamma^*)<-V_0$ and $b(a^*,\gamma^*)<U_0$, since $\displaystyle\lim_{a\rightarrow0}b(a,\gamma)=0$ by Lemma \ref{lba} (ii) and $d(a,\gamma)$ is monotonically decreasing in $a$ by Lemma \ref{lda} (i). It then follows that $(a^*,\gamma^*)\in\Gamma_2\cap\Gamma_4$, so $\Gamma_2\cap\Gamma_4\ne\emptyset$.

 Therefore $\Gamma_1\cap\Gamma_3$ and $\Gamma_2\cap\Gamma_4$ are non-empty and disjoint. Therefore, $\Lambda_1\cap\Lambda_2$ is disconnected since it is the union of non-empty, disjoint and open sets. \end{proof}

Now Lemma \ref{top} yields that there is a point $(\bar a,\bar\gamma)\in(0,\infty)\times(0,\infty)$ which is not in $\Lambda_1\cup\Lambda_2$. Hence $b(\bar a,\bar\gamma)=U_0$ since $(\bar a,\bar\gamma)\notin\Gamma_1\cup\Gamma_2$ and $d(\bar a,\bar\gamma)=-V_0$ since $(\bar a,\bar\gamma)\notin\Gamma_3\cup\Gamma_4$. The result then follows from Theorem \ref{tj}. \end{proof}

\section{Half-line case: self-similar solutions with $\varepsilon=0$}
Now we consider
\begin{align}
-\frac{1}{2}\eta f'(\eta)=[\phi'(f(\eta))f'(\eta)]',\quad  \eta<a,
\label{l}
\end{align}
with boundary conditions
\begin{align}
  &f(0)=U_0,\\
  &\lim_{\eta\nearrow a}f(\eta)=0,\quad \lim_{\eta\nearrow a}\phi'(f(\eta))f'(\eta)=-\frac{aV_0}{2}.\label{kb6}
\end{align}
In Lemma \ref{lkb1}, we showed that for each $a>0,\gamma>0$, there exists solution $f$ for $\eta\in(a-\delta,a)$ for some $\delta>0$. For $\varepsilon=0$, we know that $f(\eta)=-V_0$ for $\eta>a$. Then with the special choice $\gamma=\frac{aV_0}{2}$, we obtain directly from Lemma \ref{lkb1} and Lemma \ref{lbw} the following proposition.
\begin{proposition}
For given $a$ and $\gamma$, there exists $\delta>0$ such that for $\eta\in(a-\delta,a)$, equation (\ref{l}) has a unique solution which is positive and satisfies the boundary condition (\ref{kb6}). This solution can be continuous back to $\eta=0$.
\label{pkb1}
\end{proposition}

The following discussion on the behaviour of $f(\eta)$ as $\eta\rightarrow 0$ is analogous to that in Lemma \ref{lba}. Note that $\gamma=\frac{aV_0}{2}$ when $\varepsilon=0$.
\begin{lemma}
$b(a):=\displaystyle\lim_{\eta\rightarrow0}f\left(\eta;a,\displaystyle\frac{aV_0}{2}\right)$ has the following properties:
\begin{itemize}
\item[{\rm (i)}] $b(a)$ is strictly monotonically increasing in $a$$;$
\item[{\rm (ii)}] $\displaystyle\lim_{a\rightarrow0} b(a)=0$$;$
\item[{\rm (iii)}] $b(a)$ is a continuous function of $a$$;$
\item[{\rm (iv)}] $\displaystyle\lim_{a\rightarrow\infty} b(a)=\infty$$;$
\end{itemize}
\label{lb0}
\end{lemma}
\begin{proof}The proof of (i) follows from the similar argument to that of Lemma \ref{lbg} (i), note that when $\varepsilon=0$, $\gamma$ is increasing in $a$.
(ii) follows from the same form of argument used to show Lemma \ref{lba} (ii).
Here $-A=\frac{aV_0}{2}+\displaystyle\frac{1}{2}\int_0^af(s){\rm d}s$ and the results follows from $-A\rightarrow0$ as $a\rightarrow 0$.
The proof of (iii) is similar to the proof of Lemma \ref{lba} (iii), note that if $a\in(a_0,a_3)$, then $\gamma\in(\frac{a_0V_0}{2},\frac{a_3V_0}{2})$.
We can obtain (iv) directly from Lemma \ref{lba} (iv) since $\gamma=\frac{aV_0}{2}>0$ by Lemma \ref{gam0}.
 \end{proof}

Now, since $\gamma$ can be express as a function of $a$, we will use one-parameter shooting to show the existence of self-similar solution.
\begin{theorem}
Suppose $\varepsilon=0$. Then there exists a unique solution $f$ of problem (\ref{jjj}).
\label{t01}
\end{theorem}
\begin{proof}Identify two``bad" sets
\begin{align*}
&S^-=\left\{a\,\big|\,b(a)<U_0\right\},\\
&S^+=\left\{a\,\big|\,b(a)>U_0\right\}.
\end{align*}
We use a shooting method. Clearly $S^-$ and $S^+$ are disjoint. By Lemma \ref{lb0} we know that $b(a)$ is monotonically increasing, then we can find $a$ large enough such that $b(a)>U_0$ since $\displaystyle\lim_{a\rightarrow \infty}b(a)=\infty$ and $a$ small enough such that $b(a)<U_0$ since $\displaystyle\lim_{a\rightarrow 0}b(a)=0$, yielding that $S^-$ and $S^+$ are non-empty. Moreover, $S^-$ and $S^+$ are open since $f$ is a continuous function of $a$. Indeed, let $a_0\in S^-$ and let $\beta:=U_0-b(a_0)$, then there exists $\mu$ such that $|b(a)-b(a_0)|<\beta$ for $|a-a_0|<\mu$ since $b(a)$ is continuous by Lemma \ref{lb0}, which implies $b(a)<b(a_0)+\beta<U_0$. A similar proof shows that $S^+$ is open.  Since $S^-$ and $S^+$ are non-empty disjoint open sets, $S^-\cup S^+\neq (0,\infty)$. Then we can conclude that there exists $a\notin S^-\cup S^+$, such that $b(a)=U_0$. \end{proof}

\section{Whole-line case: self-similar solution for the limit problem}
Similarly to the half-line case, we first state the free boundary problem. The following is the definition of the weak solution of the limit problem (\ref{c2}). The uniqueness of the weak solution is proved in \cite{selfnon}.
\begin{definition}
A function $w$ is a weak solution of problem (\ref{c2}) if
\begin{itemize}
\item[{\rm(i)}]$w\in L^\infty(Q_T)$,
\item[{\rm(ii)}]$\mathcal{D}(w)\in\mathcal{D}(\hat w)+L^{2}(0,T;W^{1,2}(\mathbb{R}))$, where $\hat w\in C^\infty(\mathbb{R})$ is a smooth function with $\hat w=U_0$ when $x<-1$ and $\hat w=-V_0$ when $x>1$,
\item[{\rm(iii)}]$w$ satisfies for all $T>0$
\begin{align}
\int_{\mathbb{R}}w_0\Psi(x,0){\rm d}x+\iint_{Q_T}w\Psi_t{\rm d}x{\rm d}t=\iint_{Q_T}\mathcal{D}(w)_x\Psi_{x}{\rm d}x{\rm d}t,
\label{wlw}
\end{align}
for all $\Psi\in\mathcal{\hat F}_T$.
\end{itemize}
\label{d122}
\end{definition}

\begin{proposition}
Let $w$ be the unique weak solution of problem (\ref{c2}). Suppose that there exists a function $\beta:[0,T]\rightarrow\mathbb{R}$ such that for each $t\in[0,T]$
\begin{align*}
w(x,t)>0\ {\rm if}\ x<\beta(t)\quad {\rm and}\quad w(x,t)<0\ {\rm if}\ x>\beta(t).
\end{align*}
Then if $t\mapsto\beta(t)$ is sufficiently smooth and the functions $u:=w^+$ and $v:=w^-$ are smooth up to $\beta(t)$, the function $u,v$ satisfy one of two limit problems, depending on whether $\varepsilon>0$ or $\varepsilon=0$. If $\varepsilon>0$, then
\begin{align}
\left\{\begin{aligned}
&u_t=\phi(u)_{xx},\quad &&{\rm in}\ \left\{(x,t)\in Q_T:x<\beta(t)\right\},\\
&v=0,\quad &&{\rm in}\ \left\{(x,t)\in Q_T:x<\beta(t)\right\},\\
&v_t=\varepsilon\phi(v)_{xx},\quad &&{\rm in}\ \left\{(x,t)\in Q_T:x>\beta(t)\right\},\\
&u=0,\quad &&{\rm in}\ \left\{(x,t)\in Q_T:x>\beta(t)\right\},\\
&\lim_{x\nearrow\beta(t)}u(x,t)=0=\lim_{x\searrow\beta(t)}v(x,t)\quad &&{\rm for}\ {\rm each}\ t\in[0,T],\\
&\lim_{x\nearrow\beta(t)}\phi[u(x,t)]_x=-\varepsilon\lim_{x\searrow\beta(t)}\phi[v(x,t)]_x\quad &&{\rm for}\ {\rm each}\ t\in[0,T],\\
&u(\cdot,0)=u_0^{\infty},\quad &&{\rm in}\ \mathbb{R}\\
&v(\cdot,0)=v_0^{\infty},\quad &&{\rm in}\ \mathbb{R},
\end{aligned}\right.
\end{align}
whereas if $\varepsilon=0$ and we suppose additionally that $\beta(0)=0$ and $t\mapsto\beta(t)$ is a non-decreasing function, then
\begin{align}
\left\{\begin{aligned}
&u_t=\phi(u)_{xx},\quad &&{\rm in}\ \left\{(x,t)\in Q_T:x<\beta(t)\right\},\\
&v=0,\quad &&{\rm in}\ \left\{(x,t)\in Q_T:x<\beta(t)\right\},\\
&v=V_0,\quad &&{\rm in}\ \left\{(x,t)\in Q_T:x>\beta(t)\right\},\\
&u=0,\quad &&{\rm in}\ \left\{(x,t)\in Q_T:x>\beta(t)\right\},\\
&\lim_{x\nearrow\beta(t)}u(x,t)=0\quad &&{\rm for}\ {\rm each}\ t\in[0,T],\\
&V_0\beta'(t)=-\lim_{x\nearrow\beta(t)}\phi[u(x,t)]_x\quad &&{\rm for}\ {\rm each}\ t\in[0,T],\\
&u(\cdot,0)=u_0^{\infty},\quad &&{\rm in}\ \mathbb{R}\\
&v(\cdot,0)=v_0^{\infty},\quad &&{\rm in}\ \mathbb{R},
\end{aligned}\right.
\end{align}
where $\beta'(t)$ denotes the speed of propagation of the free boundary $\beta(t)$ and we suppose that $\beta(0)=0$ and $t\mapsto\beta(t)$ is a non-decreasing function.
\end{proposition}

\subsection{Preliminaries for self-similar solutions}
The following results will be used to prove that if there is a self-similar solution, then it is a weak solution of (\ref{c2}) and will be useful in Section 6.
\begin{lemma}\label{pfb}
If $f$ satisfies (\ref{k3}) and boundary condition (\ref{kb22}), (\ref{kb11}), then for $\eta<\min\{a,0\}$ we have
\begin{align}
\phi(U_0)-\phi(f(\eta))\le G\int_{-\infty}^\eta e^{\frac{-s^2}{4\phi'(U_0)}}{\rm d}s,
\end{align}
where
\begin{align*}
G=\left\{\begin{aligned}-&\phi'(f(0))f'(0),\quad &&{\rm if}\ a>0,\\&\gamma,&&{\rm if}\ a\le0.\end{aligned}\right.
\end{align*}
\end{lemma}
\begin{proof}Denote $N=\phi'(U_0)$. We have $\phi'(f)<\phi'(U_0)$ since $f$ is monotonically decreasing by Lemma \ref{mono} and $\phi'$ is increasing. Then we get directly from the equation of $f$ for $\eta>a$ that
\begin{align*}
\frac{\eta}{2N}[\phi(f)]'\ge-[\phi(f)]'',
\end{align*}
then multiplying by $e^{\frac{\eta^2}{4N}}$, we get
\begin{align}
\left\{e^{\frac{\eta^2}{4N}}[\phi(-f)]'\right\}'\ge 0,\label{ve111}
\end{align}
when $a\le0$, integrating from $\eta$ to $a$ yields
\begin{align}
[\phi(f(\eta))]'\le Ce^{\frac{-\eta^2}{4N}},\label{ve1}
\end{align}
where $C=\gamma$. Then integrating from $\eta$ to $\infty$ we get
\begin{align*}
\phi(U_0)-\phi(f(\eta))\le C\int_{-\infty}^\eta e^{\frac{-s^2}{4\phi'(U_0)}}{\rm d}s.
\end{align*}
When $a>0$, the result follows by integrating (\ref{ve111}) from $\eta$ to $0$.
\end{proof}

We can obtain similar estimates to those in Lemma \ref{pfb} for comparison of $\phi(f)$ to $\phi(V_0)$ as $\eta\rightarrow\infty$. Then we have the following corollary.
\begin{corollary}
If f satisfies (\ref{k3}) and boundary condition (\ref{kb22}), (\ref{kb11}), then $f$ converges to $U_0,-V_0$ exponentially as $\eta$ tends to $-\infty, \infty$.
\label{cf}
\end{corollary}
\begin{proof}We know from Lemma \ref{pfb} that, for $\eta<0$
\begin{align}
\phi'(s)\big(U_0-f(\eta)\big)\le G\int_{-\infty}^\eta e^{\frac{-s^2}{4\phi'(U_0)}}{\rm d}s,
\label{pm1}
\end{align}
 for some $\eta$ such that $f(\eta)<s<U_0$. Since $f(\eta)\rightarrow U_0$ as $\eta\rightarrow-\infty$, there exists a $\eta_0<-1$ such that $f(\eta)>\frac{U_0}{2}$ as $\eta<\eta_0$. Then we have if $\eta<\eta_0$, $\phi'(s)>\phi'(\frac{U_0}{2})$, because $\frac{U_0}{2}<\eta<U_0$. It follows from (\ref{pm1}) that for $\eta<\eta_0$
\begin{align*}
U_0-f(\eta)\le\frac{G}{\phi'(\frac{U_0}{2})}\int_{-\infty}^\eta e^{\frac{-s^2}{4\phi'(U_0)}}{\rm d}s\le Ke^{\frac{\eta}{4\phi'(U_0)}},
\end{align*}
where $K=\frac{4G\phi'(U_0)}{\phi'(\frac{U_0}{2})}$.

The proof for $f$ converges to $-V_0$ exponentially as $\eta\rightarrow\infty$ can be proved similarly.
\end{proof}

The main results of this paper in whole-line case are Theorem \ref{j21}, where $\varepsilon>0$, and Theorem \ref{j22}, where $\varepsilon=0$. The results follow similar arguments to those used in the half-line case. The existence of self-similar solutions of Problem (\ref{j2}) and (\ref{jjj2}) are proved in Theorem \ref{tj2} and \ref{tew0}.
\begin{theorem}
The unique weak solution $w$ of problem (\ref{c2}) with $\varepsilon>0$ has a self-similar form. There exists a function $f:\mathbb{R}\mapsto\mathbb{R}$ and a constant $a\in\mathbb{R}$ such that
\begin{align*}
w(x,t)=f(\frac{x}{\sqrt{t}}),\ (x,t)\in Q_T\ {\rm and}\ \beta(t)=a\sqrt{t},\ t\in[0,T].
\end{align*}
 Denote $\eta=\displaystyle\frac{x}{\sqrt{t}}$, $f$ satisfies the system
\begin{align}
\left\{\begin{aligned}
  &-\frac{1}{2}\eta f'(\eta)=[\phi'(f(\eta))f'(\eta)]',\quad &&{\rm if}\ \eta<a,\\
  &-\frac{1}{2}\eta f'(\eta)=[\varepsilon\phi'(-f(\eta))f'(\eta)]',\quad &&{\rm if}\ \eta>a,\\
  &\lim_{\eta\rightarrow -\infty}f(\eta)=U_0,\quad \lim_{\eta\rightarrow \infty}f(\eta)=-V_0,\\
  &\lim_{\eta\nearrow a}f(\eta)=0=-\lim_{\eta\searrow a}f(\eta),\\
  &\lim_{\eta\nearrow a}\phi'(f(\eta))f'(\eta)=\varepsilon\lim_{\eta\searrow a}\phi'(-f(\eta))f'(\eta).\\
\end{aligned}\right.
\label{j2}
\end{align}
where a prime denotes differentiation with respect to $\eta$.
\label{j21}
\end{theorem}

\begin{theorem}
The unique weak solution $w$ of problem (\ref{c2}) with $\varepsilon=0$ has a self-similar form. There exists a function $f:\mathbb{R}\mapsto\mathbb{R}$ and a constant $a\in\mathbb{R^+}$ such that
\begin{align*}
w(x,t)=f(\frac{x}{\sqrt{t}}),\ (x,t)\in Q_T\ {\rm and}\ \beta(t)=a\sqrt{t},\ t\in[0,T].
\end{align*}
 Denote $\eta=\displaystyle\frac{x}{\sqrt{t}}$, $f$ satisfies the system
\begin{align}
\left\{\begin{aligned}
  &-\frac{1}{2}\eta f'(\eta)=[\phi'(f(\eta))f'(\eta)]',\quad &&{\rm if}\ \eta<a,\\
  &f(\eta)=-V_0,\quad &&{\rm if}\ \eta>a,\\
  &\lim_{\eta\rightarrow -\infty}f(\eta)=U_0,\\
  &\lim_{\eta\nearrow a}f(\eta)=0,\\
  &\lim_{\eta\nearrow a}\phi'(f(\eta))f'(\eta)=-\frac{aV_0}{2},\\
\end{aligned}\right.
\label{jjj2}
\end{align}
where a prime denotes differentiation with respect to $\eta$.
\label{j22}
\end{theorem}

We now study the existence of a solution $f$ that satisfies (\ref{j2}) by using the similar argument to that of half-line case.  The main difference with the half-line case is that now we need to consider $\eta\in\mathbb{R}$ and investigate the case when $a\le0$ in addition to $a>0$.

We start with some preliminary results that will be used later.
The monotonicity of $f$ follows from arguments analogous to those in the proof of Lemma \ref{mono}.
\begin{lemma}
Suppose $\varepsilon>0$. If $f$ satisfies (\ref{j2}), then $f'(\eta)<0$ for all $\eta\ne a$.\label{mono1}
\end{lemma}

Now we prove $\gamma$ is strictly positive when $\varepsilon>0$. Recall that when $\varepsilon>0$
\begin{align*}
\gamma=-\displaystyle\lim_{\eta\nearrow a}\phi'(f(\eta))f'(\eta)=-\varepsilon\displaystyle\lim_{\eta\searrow a}\phi'(-f(\eta))f'(\eta).
\end{align*}
\begin{lemma}\label{gam1}
Suppose $\varepsilon>0$. Let $f$ be a solution of (\ref{j2}), then $\gamma>0$.
\end{lemma}
\begin{proof}Suppose $\gamma\le0$, we consider in two cases, $a\ge0$ and $a\le0$. When $a\ge0$, the proof is the same to the proof of Lemma \ref{ga2}.

Now we let $a\le0$. Integrating the equation for $\eta<a$ in (\ref{j2}) from $\eta$ to $a$ yields
 \begin{align}
 -\frac{1}{2}\int^a_\eta sf'(s){\rm d}s=-\phi'(f(\eta))f'(\eta)-\gamma.\label{wlsg2}
 \end{align}
 The left-hand side of (\ref{wlsg2}) is negative since $\eta<0$ and $f'(\eta)<0$ by Lemma \ref{mono1} whereas the right-hand side of (\ref{wlsg2}) is positive if $\gamma\le0$ since $f'(\eta)<0$. Therefore, it follows by contradiction that $\gamma>0$.
\end{proof}

The following lemma proves the analogous result for $\gamma$ when $\varepsilon=0$. Recall that in this case, $\gamma=\displaystyle\lim_{\eta\searrow a}\phi'(f(\eta))f'(\eta)=\frac{aV_0}{2}$.
\begin{lemma}
Suppose $\varepsilon=0$ and let $f$ be a solution of (\ref{jjj2}). Then $a,\gamma>0$.\label{lwg0}
\end{lemma}
\begin{proof}We know from the proof of Lemma \ref{gam1} that $\gamma>0$ when $a\le0$, since the proof when $a\le0$ only involved the equation (\ref{wlsg2}) for $\eta<a$. However, when $\varepsilon=0$, the fact that $\gamma=\frac{aV_0}{2}>0$ contracts $a\le0$. In conclusion, if $f$ satisfies (\ref{jjj2}) when $\varepsilon=0$, both $a$ and $\gamma$ are positive.\end{proof}

\subsection{Self-similar solutions with $\varepsilon>0$}
First we consider $f$ that satisfies the equation
\begin{align}
-\frac{1}{2}\eta f'(\eta)=[\phi'(f(\eta))f'(\eta)]',\quad  \eta<a.
\label{k3}
\end{align}
At the boundaries we require
\begin{align}
&\lim_{\eta\rightarrow -\infty}f(\eta)=U_0,\label{kb22}\\
&\displaystyle\lim_{\eta\nearrow a}f(\eta)=0,\quad \displaystyle\lim_{\eta\nearrow a}\phi'(f(\eta))f'(\eta)=-\gamma,\label{kb11}
\end{align}
where $a$ and $\gamma>0$ are constant.

We can obtain various results when $a<0$ from the half-line problem by a change of variables. We define
\begin{align}
-g(-\eta):=f(\eta).
\label{cov}
\end{align}
 Denoting $\hat a=-a$ and $\hat\eta=-\eta$, we get
 \begin{align*}
 -\frac{1}{2}\hat\eta g'(\hat\eta)=[\phi'(-g(\hat\eta))g'(\hat\eta)]'.
 \end{align*}

 The following result is immediate from Lemma \ref{ld0}, by using the change of variables (\ref{cov}).
\begin{lemma}
If $f$ satisfies (\ref{k3}) and the boundary conditions (\ref{kb22}) and (\ref{kb11}), then we have the derivative of $f$ vanishes as $\eta\rightarrow -\infty$
\begin{align*}
\displaystyle\lim_{\eta\rightarrow -\infty}f'(\eta)=0.
\end{align*}
\label{ld1}
\end{lemma}

From the previous results in half-line case, we know immediately that: the solution $f$ exists is unique locally in a left-neighbourhood $(a-\delta,a)$ of $a$ when $a>0$, and it is monotonically decreasing. By the change of variables (\ref{cov}), we know from Lemma \ref{lkb2} and \ref{lkb3} that solution $f$ is unique locally in a left-neighbourhood $(a-\delta,a)$ of $a$ when $a<0$.

We therefore have the following lemma, for which it remains to prove the local existence and uniqueness of the solution $f$ when $a=0$.
\begin{lemma}
For given $a\in\mathbb{R}$ and $\gamma>0$, there exists $\delta>0$ such that in $(a-\delta,a)$ equation (\ref{k2}) has a unique solution which is positive and satisfies the boundary condition (\ref{kb11}).
\label{lkb11}
\end{lemma}
\begin{proof}The proof for $a>0$ is similar to the proof of Lemma \ref{lkb1} and \ref{lbw}. If $a\le0$, by using the similar approach to that of Lemma \ref{lkb1}, writing $\eta=\sigma(f)$, then
\begin{align}
\sigma(f)=-2\int_0^f\frac{\phi'(\theta)}{\int_0^\theta \sigma(s){\rm d}s+2\gamma}{\rm d}\theta,\label{wls2}
\end{align}
and if we set
\begin{align*}
\t(f)=-\sigma(f)=-\eta,
\end{align*}
then (\ref{wls2}) becomes
\begin{align}
\t(f)=2\int_0^f\frac{\phi'(\theta)}{-\int_0^\theta \t(s){\rm d}s+2\gamma}{\rm d}\theta.
\end{align}
Now we denote by $X$ the set of continuous functions $\t(f)$ on $[0,\mu]$, satisfying $0\le\t(f)\le\frac{1}{2}$, and $\Vert\cdot\Vert$ the supremum norm on $X$. Then $X$ is a complete metric space. Choose $\mu$ small enough that $\mu<2\gamma$, on $X$ we introduce the map
\begin{align*}
M(\t)(f)=2\int_0^f\frac{\phi'(\theta)}{-\int_0^\theta \t(s){\rm d}s+2\gamma}{\rm d}\theta\le 2\int_0^\mu\frac{\phi'(\theta)}{\gamma}{\rm d}\theta.
\end{align*}
It is clear that $M(\t)(f)$ is well-defined, non-negative and continuous. Moreover, $M(\t)(f)\le\frac{1}{2}$ if
\begin{align}
\int_0^\mu\frac{\phi'(\theta)}{\gamma}{\rm d}\theta\le\frac{1}{4}. \label{wls3}
\end{align}
Therefore, if $\mu$ is chosen small enough that (\ref{wls3}) is satisfied, $M$ maps $X$ into itself.

We wish to ensure that $M$ is a contraction map, so let $\t_1,\t_2\in X$, we have for chosen $\mu<2\gamma$

\begin{align*}
\Vert M(\t_1)-M(\t_2)\Vert \le&4\int_0^\mu\frac{\phi'(\theta)}{\gamma}{\rm d}\theta\Vert \t_1-\t_2\Vert,
\end{align*}
and it follows that $M$ is a contraction map if
\begin{align*}
4\int_0^\mu\frac{\phi'(\theta)}{\gamma}{\rm d}\theta< 1.
\end{align*}
This constitutes our third restriction on $\mu$, which implies the first one (\ref{wls3}). The result follows from a contraction mapping principle \cite{line}. \end{proof}

We know that if $a>0$, the local solution in Lemma \ref{lkb11} can be continued back to $\eta=0$ by Lemma \ref{lbw} in the half-line case. It can be shown that $f'(0)<C$ \cite{thesis}, which ensures $f$ can be continued back a little bit from $0$ by Picard's theorem and the fact that $f(0)>0$. In order to show that the unique local solution in Lemma \ref{lkb11} can be continued back to $\eta=-\infty$ by the Global Picard Theorem, we will prove some estimates for $-f'$ and $f$.

The following lemma proves the boundedness for $f$ in three cases: $a>0$, $a=0$ and $a<0$.

\begin{lemma}
If $f$ satisfies (\ref{k3}) and the boundary conditions (\ref{kb22}) and (\ref{kb11}), then we have for fixed $a,\gamma$, there exists $K>0$ such that $0<f(\eta)<K$ for all $\eta<a$.
\label{lfb}
\end{lemma}
\begin{proof}
\noindent{\bf Case 1. } For $a>0$, first we consider $\eta\in[0,a)$, from (\ref{k3}) we know that
\begin{align}
\frac{-\phi'(f(\eta))f'(\eta)}{\frac{2\gamma}{a}+f(\eta)}\le\frac{a}{2},
\label{fa2}
\end{align}
integrating (\ref{fa2}) from $0$ to $\eta$ gives
\begin{align*}
\int^{f(0)}_0\frac{\phi'(s)}{\frac{2\gamma}{a}+f(s)}{\rm d}s\le\frac{a^2}{2},
\end{align*}
then $f(\eta)\le f(0)$ is bounded for $0\le\eta<a$, since $f$ is monotonic decreasing in $\eta$.

Next we consider $\eta\in[-2\rho,0)$ for some positive $\rho$.
Integrating (\ref{k3}) from $\eta$ to $0$ we have
\begin{align}
\phi'(f(0))f'(0)-\phi'(f(\eta))f'(\eta)\le 0,
\label{fa3}
\end{align}
then integrating (\ref{fa3}) yields
\begin{align*}
\phi(f(\eta))\le \phi(f(0))+2\rho\phi'(f(0))f'(0).
\end{align*}

Then for $\eta<-\rho$, integrating (\ref{k3}) from $\eta$ to $-\rho$ we get
\begin{align*}
\frac{\rho}{2}f(\eta)-\phi'(f(\eta))f'(\eta)\le \frac{\rho}{2}f(-\rho)-\phi'(f(-\rho))f'(-\rho),
\end{align*}
since $f'<0$. Therefore $f(\eta)\le K$ for fixed $a,\gamma$ and all $\eta<a$.

\noindent{\bf Case 2. }For $a=0$, the same proof can be used as in $a>0$ case, with $\phi(f(\eta))\le 2\rho\gamma$ for $\eta\in[-2\rho,a)$.

\noindent{\bf Case 3. }For $a<0$, if $\eta\in[-2\rho, a)$, integrating (\ref{k3}) from $\eta$ to $a$ we have
\begin{align}
-(\phi(f(\eta)))'<\gamma,
\label{fa4}
\end{align}
then integrating (\ref{fa4}) from $\eta$ to $a$ gives
\begin{align*}
\phi(f(\eta))\le(a+2\rho)\gamma.
\end{align*}
For $\eta<-\rho$ the proof is the same as $a>0$.\end{proof}

Next, we prove the boundedness for $-f'$ in three cases: $a>0$, $a=0$ and $a<0$.
\begin{lemma}
If $f$ satisfies (\ref{k3}) and the boundary conditions (\ref{kb22}) (\ref{kb11}), then for fixed $a,\gamma$, there exists $\hat K$ such that $0<-\phi'(f(\eta))f'(\eta)<\hat K$ for all $\eta<a$.
\label{lfp}
\end{lemma}
\noindent{\bf Proof. }
For all $a\in\mathbb{R}$, first consider $\eta\in[-2\rho,a)$ for some $\rho>0$. Integrating (\ref{k3}) from $\eta$ to $a$ we have that $-\phi'(f(\eta))f'(\eta)\le\gamma+K(a+\rho)$, where $K$ is positive constant that $f(\eta)\le K$, by Lemma \ref{lfb}. Then for $\eta\le-2\rho$, we know that $-\phi'(f(\eta))f'(\eta)\le-\phi'(f(-2\rho))f'(-2\rho)\le \hat K$ for fixed $a,\gamma$.\hfill $\B$ \\

The following lemma is a consequence of Lemma \ref{lfb}, Lemma \ref{lfp} together with \cite[Theorem 1.186]{odea}.

\begin{lemma}
For given $a,\gamma$, the unique local solution in Lemma \ref{lkb11} can be continued back to $\eta=-\infty$.
\label{lbw2}
\end{lemma}

~\\
Now define
\begin{align*}
b(a,\gamma):=\displaystyle\lim_{\eta\rightarrow-\infty}f(\eta;a,\gamma),
\end{align*}
where $\gamma:=-\displaystyle\lim_{\eta\nearrow a}\phi'(f(\eta))f'(\eta)$ with $\gamma>0$.
Note that we use the same notation $b(a,\gamma)$ as in the half-line case, but here $b(a,\gamma)$ define as the function of $f(\eta;a,\gamma)$ as $\eta\rightarrow-\infty$ rather than $\eta\rightarrow0$.

We can obtain from Corollary \ref{ccf}, Lemma \ref{lfc} and the change of variables (\ref{cov}) that $f$ is a continuous function of $a$ and $\gamma$.
\begin{lemma}\label{lfc1}
For each fixed $\eta^*<a$, if $f$ satisfies (\ref{k3}) and (\ref{kb11}), then
\begin{itemize}
\item[{\rm (i)}] $f(\eta^*;a,\gamma)$ is a continuous function of $\gamma$ for fixed $a$;
\item[{\rm (ii)}] $f(\eta^*;a,\gamma)$ is a continuous function of $a$ for fixed $\gamma$.
\end{itemize}
\end{lemma}

The following corollary follows directly from Lemma \ref{lfc1} and Corollary \ref{cf} as $\eta\rightarrow-\infty$.
\begin{corollary}
If $f$ satisfies (\ref{k3}) and (\ref{kb11}), then
\begin{itemize}
\item[{\rm (i)}] $b(a,\gamma)$ is a continuous function of $\gamma$ for fixed $a$;
\item[{\rm (ii)}] $b(a,\gamma)$ is a continuous function of $a$ for fixed $\gamma$.
\end{itemize}
\label{cfc1}
\end{corollary}

Now we consider $f$ satisfying the equation
\begin{align}
-\frac{1}{2}\eta f'(\eta)=[\varepsilon\phi'(-f(\eta))f'(\eta)]',\quad  \eta>a.
\label{k4}
\end{align}
At the boundaries we require
\begin{align}
&\lim_{\eta\rightarrow \infty}f(\eta)=V_0,\label{kb23}\\
&\displaystyle\lim_{\eta\searrow a}f(\eta)=0,\quad \displaystyle\lim_{\eta\searrow a}\varepsilon\phi'(-f(\eta))f'(\eta)=-\gamma,\label{kb12}
\end{align}
where $a$ and $\gamma>0$ are constants.

Following from what we studied on the positive solution, we can directly obtain the local existence, uniqueness results and continuity forward to $\eta=\infty$ of the solution $f$ by the change of variables (\ref{cov}). As for $\eta<a$, we know that $f$ is a monotonically decreasing function. Moreover, we know from Lemma \ref{ld0} directly that $\displaystyle\lim_{\eta\rightarrow\infty}f'(\eta)=0$. Similarly, if we define
\begin{align*}
d(a,\gamma):=\displaystyle\lim_{\eta\rightarrow \infty}f(\eta;a,\gamma).
\end{align*}

Next, we discuss the properties of $b(a,\gamma)$ and $d(a,\gamma)$. When we study the properties of $b(a,\gamma)=\displaystyle\lim_{\eta\rightarrow-\infty}f(\eta;a,\gamma)$, we can see the properties of $d_{\mathbb{R}^+}(a,\gamma)=\displaystyle\lim_{\eta\rightarrow\infty}f_{\mathbb{R}^+}(\eta;a,\gamma)$ in half-line case. The properties of $d(a,\gamma)$ are obtained immediately using the change of variables (\ref{cov}).
\begin{lemma}
The functions $b(a,\gamma)=\displaystyle\lim_{\eta\rightarrow-\infty}f(\eta;a,\gamma)$ and $d(a,\gamma)=\displaystyle\lim_{\eta\rightarrow\infty}f(\eta;a,\gamma)$ satisfy the analogous properties to those in Lemma \ref{lbg}, \ref{lba}, \ref{lda}, \ref{ldg}, \ref{lc1} and \ref{lcag}, where the property $\displaystyle\lim_{a\rightarrow0}b(a,\gamma)=0$ is replaced by $\displaystyle\lim_{a\rightarrow-\infty}b(a,\gamma)=0$.
\label{lba2}
\end{lemma}
\begin{proof}
Consider $a<0$, integrating (\ref{k3}) from $\eta$ to $a$ we get
\begin{align*}
-\phi'(f(\eta))f'(\eta)\le\gamma-\frac{a}{2}\int_\eta^af'(s){\rm d}s=\frac{a}{2}\left(\frac{2\gamma}{a}+f(\eta)\right),
\end{align*}
then the result follows from the fact that $f'<0$, gives $\frac{2\gamma}{a}+f<0$.
\end{proof}

Similarly to the half-line case, a two-parameter shooting method can be used to prove the existence of a self-similar solution of problem (\ref{j2}).

\begin{theorem}
Suppose $\varepsilon>0$, then there exists a unique solution $f$ of problem (\ref{j2}).\label{tj2}
\end{theorem}

\begin{proof}This follows by using a similar argument to that in the proof of Lemma \ref{lcon} in the half-line case, applying Lemma \ref{top} to the set $\mathbb{R}\times(0,\infty)$, which is homeomorphic to the entire plane, for example, if we define a homeomorphism $g: \mathbb{R}\times(0,\infty)\mapsto\mathbb{R}^2$ such that $g(x,y)=(x,\log y)$, together with Theorems \ref{j21}. \end{proof}

\subsection{Self-similar solutions for $\varepsilon=0$}
Now we consider
\begin{align}
-\frac{1}{2}\eta f'(\eta)=[\phi'(f(\eta))f'(\eta)]',\quad  \eta<a,
\label{l234}
\end{align}
with boundaries required
\begin{align}
  &\lim_{\eta\rightarrow-\infty}f(\eta)=U_0,\non\\
  &\lim_{\eta\nearrow a}f(\eta)=0, \quad\lim_{\eta\nearrow a}\phi'(f(\eta))f'(\eta)=-\frac{aV_0}{2}.\label{l345}
\end{align}
For $\varepsilon=0$ case, we know that $f(\eta)=-V_0$ for $\eta>a$ and $\gamma=\frac{aV_0}{2}$ which are the same as in the half-line case. Note that when $\varepsilon=0$,  $a,\gamma$ are positive by Lemma \ref{lwg0}. Since we showed that for each $a\in \mathbb{R},\gamma>0$, there exists solution $f$ for $\eta\in(a-\delta,a)$ for some $\delta>0$, then there exists a solution of (\ref{l234}) on interval $(a-\delta,\a)$. By Lemma \ref{lbw2} we know that with $\gamma=\frac{aV_0}{2}$ the solution $f$ can be continued back to $-\infty$. As for the $\varepsilon>0$ case, we know that $f$ is a monotonically decreasing function.

Now define
\begin{align*}
b(a)\colon=\displaystyle\lim_{\eta\rightarrow-\infty}f\left(\eta;a,\displaystyle\frac{aV_0}{2}\right).
\end{align*}
Note that we use the same notation $b(a)$ as in the half-line case, but here $b(a)$ is defined as the limit of $f(\eta;a)$ as $\eta\rightarrow-\infty$ rather than $\eta\rightarrow0$.
\begin{lemma}\label{lb01}
The function $b(a)=\displaystyle\lim_{\eta\rightarrow-\infty}f(\eta;a)$ satisfies the same properties as in Lemma \ref{lb0}.
\end{lemma}
\begin{proof}(i)(iii)(iv) follow immediately from Lemma \ref{lb0}. It remains to prove that $\displaystyle\lim_{a\rightarrow0}b(a)=0$, because $b(a)$ is the limit of $f$ at $\eta\rightarrow-\infty$ rather than $\eta=0$. Note that $a>0$ when $\varepsilon=0$, since $\gamma=\frac{aV_0}{2}>0$.

 Let $a<1$ and denoting $N=\phi'(b(1))$, we have $\phi'(f)\le N$ by (i). Then we get directly from (\ref{l234}) that
\begin{align}
-\frac{\eta}{2N}[\phi(f(\eta))]'\ge[\phi(f(\eta))]'',\label{1223}
\end{align}
then multiplying (\ref{1223}) by $e^{\frac{-\eta^2}{4N}}$ and integrating from $\eta$ to $0$ yields
\begin{align*}
[\phi(f(\eta))]'\ge Ae^{\frac{-\eta^2}{4N}},
\end{align*}
where $A=\phi'(f(0))f'(0)<0$. Integrating again from $-\infty$ to $0$ we get
\begin{align*}
\phi(b(a))\le\phi(f(0))-A\int_{-\infty}^0e^{\frac{-s^2}{4N}}{\rm d}s.
\end{align*}
Integrating the equation (\ref{l234}) from $0$ to $a$ yields
\begin{align*}
\frac{1}{2}\int_0^af(s){\rm d}s=\frac{aV_0}{2}-\phi'(f(0))f'(0).
\end{align*}
Then we have
\begin{align*}
-A=\frac{aV_0}{2}+\displaystyle\frac{1}{2}\int_0^af(s){\rm d}s\rightarrow 0\ {\rm as}\ a\rightarrow 0.
\end{align*}
Therefore $\phi(b(a))\le\phi(f(0))-A\displaystyle\int^0_{-\infty}e^{\frac{-s^2}{4N}}{\rm d}s\rightarrow 0$ as $a\rightarrow 0$, since $\displaystyle\int^0_{-\infty}e^{\frac{-s^2}{4N}}<\infty$ and $\phi(f(0))\rightarrow 0$ as $a\rightarrow 0$ by Lemma \ref{lba} (ii) and $\phi(0)=0$.
\end{proof}

The following result follows by using a one-parameter shooting method similar to that used to prove Theorem \ref{t01}, replacing $(0,\infty)$ with $\mathbb{R}$.
\begin{theorem}
Suppose $\varepsilon=0$, then there exists a unique solution $f$ of problem (\ref{jjj2}).
\label{tew0}
\end{theorem}

\section{Self-similar solutions with special $\phi' (f)=f^{m-1}$ with $m>1$}
The choice of $\phi$ satisfying (\ref{phi}) and (\ref{phi1}) plays an important role in the characterisation of rates at which one substance invades another of the system (\ref{a2}). For concreteness, we consider the specific family that is motivated by porous medium equation
\begin{align}
\phi'(w)=w^{m-1},\label{pm}
\end{align}
with $m>1$, which satisfies the conditions (\ref{phi}) and (\ref{phi1}).

The form of self-similar solution of the limit problems with nonlinear diffusion $w(x,t)=f(\eta)$ is exactly the same as in the linear diffusion case where $\eta=\frac{x}{\sqrt t}$ is independent of the choice of $\phi$. We are interested in how the free boundary is affected by $m$, in the other words, the relationship between $m$ and $a$, where $a$ gives the position of free-boundary because $f(a)=0$.

In the following section, we focus on the whole line case with $\varepsilon=0$ and explore the self-similar solution $f_m(\eta)=f(\eta;m)$, in particular, how the value $a$ depends on $m$. The study on half-line case and when $\varepsilon>0$ can be found on \cite{thesis}.

We consider the whole line case with the specific choice of $\phi'$ (\ref{pm}). For $\varepsilon=0$, the problem satisfied by $f$ is

\begin{align}
\left\{\begin{aligned}
  &-\frac{1}{2}\eta f'(\eta)=[f^{m-1}(\eta)f'(\eta)]',\quad &&{\rm if}\ -\infty<\eta<a,\\
  &f(\eta)=-V_0,\quad &&{\rm if}\ a<\eta<\infty,\\
  &\lim_{\eta\rightarrow -\infty}f(\eta)=U_0,\\
  &\lim_{\eta\nearrow a}f(\eta)=0,\\
  &\lim_{\eta\nearrow a}f^{m-1}(\eta)f'(\eta)=-\frac{aV_0}{2},\\
\end{aligned}\right.
\label{j234}
\end{align}
where $a$ is positive.

Recall $f_{m_i}(\eta)=f(\eta;m_i)$, denote $a_{m_i}$ be the position of free boundary where $f_{m_i}(a_{m_i})=0$, and $\gamma_{m_i}=-\displaystyle\lim_{\eta\nearrow a_{m_i}}f_{m_i}^{m_i-1}(\eta)f_{m_i}'(\eta)=\frac{a_{m_i}V_0}{2}$.

Consider $f_{m_1}$ and $f_{m_2}$ satisfying (\ref{j234}) with $m_1\neq m_2$, we first deduce some results about intersection of $f_{m_1}$ and $f_{m_2}$.

\begin{lemma}
Suppose $a_{m_1}<a_{m_2}$, if $f_{m_1}$ and $f_{m_2}$ satisfy (\ref{j234}), then there exists some $\eta_0<a_{m_1}$ such that $f_{m_1}(\eta_0)=f_{m_2}(\eta_0)$.
\label{lm1}
\end{lemma}
\begin{proof}
For $\varepsilon=0$, suppose there exists no $\eta_0<a_{m_1}$ such that $f_{m_1}(\eta_0)=f_{m_2}(\eta_0)$. Then we must have $f_{m_1}<f_{m_2}$ for all $\eta\in\mathbb{R}$ since $a_{m_1}<a_{m_2}$.

We consider
\begin{align}
&-\frac{1}{2}\eta f'(\eta)=[f^{m-1}(\eta)f'(\eta)]',\quad  \eta<a,\label{j241}
\end{align}
with $\gamma=\frac{aV_0}{2}$.
If $f_{m_1}$ and $f_{m_2}$ are solution of (\ref{j241}) with corresponding $m_1, m_2$, then, integrating the equation of $f$ from $\eta$ to $a_{m_1}, a_{m_2}$, subtracting the equations and letting $\eta\rightarrow-\infty$ yields
\begin{align*}
\frac{1}{2}\int_{-\infty}^{a_{m_1}}\left[f_{m_2}(s)-f_{m_1}(s)\right]{\rm d}s+\int_{a_{m_1}}^{a_{m_2}}f_{m_2}(s){\rm d}s=\frac{a_{m_1}V_0}{2}-\frac{a_{m_2}V_0}{2}.
\end{align*}
We know that the left-hand side is positive since $f_{m_2}>f_{m_1}$ for $\eta<a_{m_1}$. For $a_{m_1}<a_{m_2}$, the left-hand side is negative, then there is a contradiction, then there must exists $\eta_0<a_{m_1}$ such that $f_{m_1}(\eta_0)=f_{m_2}(\eta_0)$. \end{proof}

We obtain the following result when $\varepsilon=0$ by exploiting the fact that $\gamma=\frac{aV_0}{2}$. In the following result, we only study the positive solutions $f(\eta)$ for $\eta<a$, and we consider additional conditions $0<U_0<1$ and $m\ge2$.  Note that the relationship between $a$ and $m$ tells us how the speed of one substance penetrating into the other is affected by $m$.
\begin{theorem}\label{tend}
Let $\varepsilon=0$ and $U_0<1$, suppose $f_{m_1},f_{m_2}$ satisfy (\ref{j234}) with corresponding $m_1,m_2\ge2$ and $a_{m_1},a_{m_2}$. Then if $m_1>m_2$, we have
\begin{align*}
 0<a_{m_1}<a_{m_2}.
\end{align*}
\end{theorem}
\begin{proof}For $\varepsilon=0$ case, we have $\gamma=\frac{aV_0}{2}>0$, then by Lemma \ref{lm1}, there exists $\eta_0<\min\{a_{m_1},a_{m_2}\}$ such that $f_{m_1}(\eta_0)=f_{m_2}(\eta_0)$. We know that $a_{m_1},a_{m_2}>0$ since $\gamma_{m_1},\gamma_{m_2}>0$.

Now let $\eta_0$ be the closet intersection point to $\min\{a_{m_1},a_{m_2}\}$, integrating (\ref{j234}) from $\eta_0$ to $a_{m_1},a_{m_2}$ we get
\begin{align}
-\frac{1}{2}\eta_0 f_{m_1}(\eta_0)+\frac{1}{2}\int_{\eta_0}^{a_{m_1}}f_{m_1}(s){\rm d}s=-\frac{a_{m_1}V_0}{2}-f_{m_1}^{m_1-1}(\eta_0)f_{m_1}'(\eta_0),\label{mm01}\\
-\frac{1}{2}\eta_0 f_{m_2}(\eta_0)+\frac{1}{2}\int_{\eta_0}^{a_{m_2}}f_{m_2}(s){\rm d}s=-\frac{a_{m_2}V_0}{2}-f_{m_2}^{m_2-1}(\eta_0)f_{m_2}'(\eta_0).\label{mm02}
\end{align}
Subtracting (\ref{mm01}) from (\ref{mm02}) we have
\begin{align}
\frac{1}{2}\int_{\eta_0}^{a_{m_1}}f_{m_1}(s){\rm d}s-\frac{1}{2}\int_{\eta_0}^{a_{m_2}}f_{m_2}(s){\rm d}s+\frac{a_{m_1}V_0}{2}-\frac{a_{m_2}V_0}{2}=f_{m_1}^{m_1-1}(\eta_0)f_{m_1}'(\eta_0)-f_{m_2}^{m_2-1}(\eta_0)f_{m_2}'(\eta_0).
\label{mma01}
\end{align}
For $m_1>m_2$ we know $f_{m_1}^{m_1-1}(\eta_0)<f_{m_2}^{m_2-1}(\eta_0)$, since $U_0<1$ and $f$ is decreasing. Then if $a_{m_1}>a_{m_2}$, the left-hand side of (\ref{mma01}) is positive and $-f_{m_1}'(\eta_0)<-f_{m_2}'(\eta_0)$, which gives
\begin{align*}
f_{m_1}^{m_1-1}(\eta_0)f_{m_1}'(\eta_0)-f_{m_2}^{m_2-1}(\eta_0)f_{m_2}'(\eta_0)<0,
\end{align*}
which contradicts the left-hand side is positive. Therefore if $m_1>m_2$, we have $a_{m_1}<a_{m_2}$.
\end{proof}

The analogous result for the half-line case can be proved by a similar method.

\end{document}